\documentclass[10pt]{amsart}

\usepackage{amsfonts,amssymb}

\usepackage{lipsum}
\usepackage{xcolor}
\usepackage{hyperref}

\input{xy}
\xyoption{all}

\newtheorem{proposition}{Proposition}[section]
\newtheorem{lemma}[proposition]{Lemma}
\newtheorem{corollary}[proposition]{Corollary}
\newtheorem{theorem}[proposition]{Theorem}
\theoremstyle{definition}

\theoremstyle{remark}
\newtheorem{remark}[proposition]{Remark}
\newtheorem{remarks}[proposition]{Remarks}\newcommand{\thlabel}[1]{\label{th:#1}}
\newcommand{\thref}[1]{Theorem~\ref{th:#1}}
\newcommand{\selabel}[1]{\label{se:#1}}
\newcommand{\seref}[1]{Section~\ref{se:#1}}

\newcommand{\lelabel}[1]{\label{le:#1}}
\newcommand{\leref}[1]{Lemma~\ref{le:#1}}
\newcommand{\prlabel}[1]{\label{pr:#1}}
\newcommand{\prref}[1]{Proposition~\ref{pr:#1}}

\newcommand{\eqlabel}[1]{\label{eq:#1}}
\newcommand{\equref}[1]{(\ref{eq:#1})}

\def\ra{\rightarrow}
\def\cd{\cdot}
\def\Id{{\rm Id}}

\def\htw{{H(2)}}

\newcommand{\alp}{\sigma}

\def\ot{\otimes}
\def\va{\varepsilon}

\def\un{\underline}

\def\mf{\mathfrak}

\def\mfq{\mf {q}}
\def\le{\langle}
\def\ri{\rangle}
\def\l{\lambda}

\def\tie{\mathrel>\joinrel\mathrel\triangleleft}

\def\va{\varepsilon}

\def\tr{\triangleright}

\def\ra{\rightarrow}
\def\a{\alpha}
\def\b{\beta}

\def\o{\omega}

\def\ov{\overline}
\def\cal{\mathcal}
\def\un{\underline}

\newcommand{\Cc}{\cal C}

\def\tie{\mathrel>\joinrel\mathrel\triangleleft}

\allowdisplaybreaks[4]

 \newcount\grcalca
 \newcount\grcalcb
 \newcount\grcalcc
 \newcount\grcalcd
 \newcount\grcalce
 \newcount\grcalcf
 \newcount\grcalcg
 \newcount\grcalch
 \newcount\grcalci
 \newcount\grrow
 \newcount\grcolumn
 \newcount\grwidth
 \newcount\factor
 \newcount\hfactor
 \newcount\qfactor
 \newcount\tfactor
 \newcount\sfactor
 \newcount\dfactor
 \hfactor = \factor
 \divide    \hfactor by 2
 \qfactor = \hfactor
 \divide    \qfactor by 2
 \tfactor = \factor
 \divide    \tfactor by 3
 \sfactor = \factor
 \divide    \sfactor by 6
 \dfactor = \factor
 \divide    \dfactor by 12

 \newcommand{\gbeg}[2]{
   \unitlength=1pt
   \grrow = #2
   \grcolumn = 0
   \grcalca = #1
   \grcalcb = #2
   \multiply \grcalca by \factor
   \grwidth = \grcalca
   \multiply \grcalcb by \factor
   \begin{minipage}{\grcalca pt}
   \begin{picture}(\grcalca,\grcalcb)
   \advance \grcalcb by -\factor
   \put(0, \grcalcb){\line(1,0){\grwidth}} }

 \newcommand{\gend}{
   \put(0, \factor){\line(1,0){\grwidth}}
   \end{picture}
   {\vskip2.5ex}
   \end{minipage} }

 \newcommand{\gnl}{
   \advance \grrow by -1
   \grcolumn = 0}

 \newcommand{\gvac}[1]{       
   \advance \grcolumn by #1} 
 
 \newcommand{\gcl}[1]{
   \grcalca = \grcolumn
   \multiply \grcalca by \factor
   \advance \grcalca by \hfactor
   \grcalcb = \grrow
   \multiply \grcalcb by \factor
   \grcalcc = #1
   \multiply \grcalcc by \factor
   \put(\grcalca,\grcalcb) {\line(0,-1){\grcalcc}} 
   \advance \grcolumn by 1}

 \newcommand{\gcn}[4]{
   \grcalca = \grcolumn
   \multiply \grcalca by \factor
   \grcalci = #3
   \multiply \grcalci by \hfactor
   \advance \grcalca by \grcalci
   \grcalcb = \grcolumn
   \multiply \grcalcb by \factor 
   \grcalci = #3
   \advance \grcalci by #4
   \multiply \grcalci by \qfactor
   \advance \grcalcb by \grcalci
   \grcalcc = \grcolumn
   \multiply \grcalcc by \factor
   \grcalci = #4
   \multiply \grcalci by \hfactor
   \advance \grcalcc by \grcalci
   \grcalcd = \grrow
   \multiply \grcalcd by \factor 
   \grcalce = \grrow
   \multiply \grcalce by \factor 
   \grcalci = #2
   \multiply \grcalci by \tfactor
   \advance \grcalce by -\grcalci
   \grcalcf = \grrow
   \multiply \grcalcf by \factor 
   \grcalci = #2
   \multiply \grcalci by \hfactor
   \advance \grcalcf by -\grcalci
   \grcalcg = \grrow
   \multiply \grcalcg by \factor 
   \grcalci = #2
   \multiply \grcalci by \tfactor
   \multiply \grcalci by 2
   \advance \grcalcg by -\grcalci
   \grcalch = \grrow
   \advance \grcalch by -#2
   \multiply \grcalch by \factor 
   \qbezier(\grcalca,\grcalcd)(\grcalca,\grcalce)(\grcalcb,\grcalcf) 
   \qbezier(\grcalcb,\grcalcf)(\grcalcc,\grcalcg)(\grcalcc,\grcalch) 
   \advance \grcolumn by #1}

 \newcommand{\gnot}[1]{
   \grcalca = \grcolumn
   \multiply \grcalca by \factor
   \advance \grcalca by \hfactor
   \grcalcb = \grrow
   \multiply \grcalcb by \factor
   \advance \grcalcb by -\hfactor
   \put(\grcalca,\grcalcb) {\makebox(0,0){$\scriptstyle #1$}} }

 \newcommand{\got}[2]{
   \grcalca = \grcolumn
   \multiply \grcalca by \factor
   \grcalcc = #1
   \multiply \grcalcc by \hfactor
   \advance \grcalca by \grcalcc
   \grcalcb = \grrow
   \multiply \grcalcb by \factor
   \advance \grcalcb by -\tfactor
   \advance \grcalcb by -\tfactor
   \put(\grcalca,\grcalcb){\makebox(0,0)[b]{$#2$}}
   \advance \grcolumn by #1}

 \newcommand{\gob}[2]{
   \grcalca = \grcolumn
   \multiply \grcalca by \factor
   \grcalcc = #1
   \multiply \grcalcc by \hfactor
   \advance \grcalca by \grcalcc
   \put(\grcalca,0){\makebox(0,0)[b]{$#2$}}
   \advance \grcolumn by #1}

 \newcommand{\gmu}{  
   \grcalca = \grcolumn
   \advance \grcalca by 1
   \multiply \grcalca by \factor
   \grcalcb = \grrow
   \multiply \grcalcb by \factor
   \grcalcc = \factor
   \advance \grcalcc by \hfactor
   \put(\grcalca,\grcalcb){\oval(\factor,\grcalcc)[b]}
   \advance \grcalcb by -\hfactor
   \advance \grcalcb by -\qfactor
   \put(\grcalca,\grcalcb) {\line(0,-1){\qfactor}} 
   \advance \grcolumn by 2}

 \newcommand{\gcmu}{   
   \grcalca = \grcolumn
   \advance \grcalca by 1
   \multiply \grcalca by \factor
   \grcalcb = \grrow
   \advance \grcalcb by -1
   \multiply \grcalcb by \factor
   \grcalcc = \factor
   \advance \grcalcc by \hfactor
   \put(\grcalca,\grcalcb){\oval(\factor,\grcalcc)[t]}
   \advance \grcalcb by \factor
   \put(\grcalca,\grcalcb) {\line(0,-1){\qfactor}} 
   \advance \grcolumn by 2}

 \newcommand{\glm}{
   \grcalca = \grcolumn
   \multiply \grcalca by \factor
   \advance \grcalca by \hfactor
   \grcalcb = \grcalca
   \advance \grcalcb by \factor
   \grcalcc = \grrow
   \multiply \grcalcc by \factor
   \grcalcd = \grcalcc
   \advance \grcalcd by -\tfactor
   \grcalce = \grcalcd
   \advance \grcalce by -\tfactor
   \put(\grcalca, \grcalcc){\line(0,-1){\tfactor}}
   \put(\grcalca, \grcalcd){\line(1,0){\factor}}
   \put(\grcalca, \grcalcd){\line(3,-1){\factor}}
   \put(\grcalcb, \grcalcc){\line(0,-1){\factor}}
   \advance \grcolumn by 2}

 \newcommand{\grm}{
   \grcalcb = \grcolumn
   \multiply \grcalcb by \factor
   \advance \grcalcb by \hfactor
   \grcalca = \grcalcb
   \advance \grcalca by \factor
   \grcalcc = \grrow
   \multiply \grcalcc by \factor
   \grcalcd = \grcalcc
   \advance \grcalcd by -\tfactor
   \grcalce = \grcalcd
   \advance \grcalce by -\tfactor
   \put(\grcalca, \grcalcc){\line(0,-1){\tfactor}}
   \put(\grcalca, \grcalcd){\line(-1,0){\factor}}
   \put(\grcalca, \grcalcd){\line(-3,-1){\factor}}
   \put(\grcalcb, \grcalcc){\line(0,-1){\factor}}
   \advance \grcolumn by 2}

 \newcommand{\glcm}{
   \grcalca = \grcolumn
   \multiply \grcalca by \factor
   \advance \grcalca by \hfactor
   \grcalcb = \grcalca
   \advance \grcalcb by \factor
   \grcalcc = \grrow
   \advance \grcalcc by -1
   \multiply \grcalcc by \factor
   \grcalcd = \grcalcc
   \advance \grcalcd by \tfactor
   \grcalce = \grcalcd
   \advance \grcalce by \tfactor
   \put(\grcalca, \grcalcc){\line(0,1){\tfactor}}
   \put(\grcalca, \grcalcd){\line(1,0){\factor}}
   \put(\grcalca, \grcalcd){\line(3,1){\factor}}
   \put(\grcalcb, \grcalcc){\line(0,1){\factor}}
   \advance \grcolumn by 2}

 \newcommand{\grcm}{
   \grcalcb = \grcolumn
   \multiply \grcalcb by \factor
   \advance \grcalcb by \hfactor
   \grcalca = \grcalcb
   \advance \grcalca by \factor
   \grcalcc = \grrow
   \advance \grcalcc by -1
   \multiply \grcalcc by \factor
   \grcalcd = \grcalcc
   \advance \grcalcd by \tfactor
   \grcalce = \grcalcd
   \advance \grcalce by \tfactor
   \put(\grcalca, \grcalcc){\line(0,1){\tfactor}}
   \put(\grcalca, \grcalcd){\line(-1,0){\factor}}
   \put(\grcalca, \grcalcd){\line(-3,1){\factor}}
   \put(\grcalcb, \grcalcc){\line(0,1){\factor}}
   \advance \grcolumn by 2}

 \newcommand{\gwmu}[1]{    
   \grcalca = \grcolumn
   \multiply \grcalca by \factor
   \grcalcd = \hfactor
   \multiply \grcalcd by #1
   \advance \grcalca by \grcalcd
   \grcalcb = \grrow
   \multiply \grcalcb by \factor
   \grcalcc = \factor
   \advance \grcalcc by \hfactor
   \grcalcd = #1
   \advance \grcalcd by -1
   \multiply \grcalcd by \factor
   \put(\grcalca,\grcalcb){\oval(\grcalcd,\grcalcc)[b]}
   \advance \grcalcb by -\hfactor
   \advance \grcalcb by -\qfactor
   \put(\grcalca,\grcalcb) {\line(0,-1){\qfactor}} 
   \advance \grcolumn by #1}

 \newcommand{\gwcm}[1]{   
   \grcalca = \grcolumn
   \multiply \grcalca by \factor
   \grcalcd = \hfactor
   \multiply \grcalcd by #1
   \advance \grcalca by \grcalcd
   \grcalcb = \grrow
   \advance \grcalcb by -1
   \multiply \grcalcb by \factor
   \grcalcc = \factor
   \advance \grcalcc by \hfactor
   \grcalcd = #1
   \advance \grcalcd by -1
   \multiply \grcalcd by \factor
   \put(\grcalca,\grcalcb){\oval(\grcalcd,\grcalcc)[t]}
   \advance \grcalcb by \factor
   \put(\grcalca,\grcalcb) {\line(0,-1){\qfactor}} 
   \advance \grcolumn by #1}

 \newcommand{\gwmuc}[1]{    
   \grcalca = \grcolumn
   \multiply \grcalca by \factor
   \advance \grcalca by \hfactor
   \grcalcb = \grrow
   \multiply \grcalcb by \factor
   \grcalcc = #1
   \advance \grcalcc by -1
   \multiply \grcalcc by \factor
   \put(\grcalca,\grcalcb){\line(1,0){\grcalcc}}
   \advance \grcalca by -\hfactor
   \grcalcd = \hfactor
   \multiply \grcalcd by #1
   \advance \grcalca by \grcalcd
   \grcalcc = \factor
   \advance \grcalcc by \hfactor
   \grcalcd = #1
   \advance \grcalcd by -1
   \multiply \grcalcd by \factor
   \put(\grcalca,\grcalcb){\oval(\grcalcd,\grcalcc)[b]}
   \advance \grcalcb by -\hfactor
   \advance \grcalcb by -\qfactor
   \put(\grcalca,\grcalcb) {\line(0,-1){\qfactor}} 
   \advance \grcolumn by #1}

 \newcommand{\gwcmc}[1]{   
   \grcalca = \grcolumn
   \multiply \grcalca by \factor
   \advance \grcalca by \hfactor
   \grcalcb = \grrow
   \multiply \grcalcb by \factor
   \advance \grcalcb by -\factor
   \grcalcc = #1
   \advance \grcalcc by -1
   \multiply \grcalcc by \factor
   \put(\grcalca,\grcalcb){\line(1,0){\grcalcc}}
   \grcalcd = #1
   \advance \grcalcd by -1
   \multiply \grcalcd by \hfactor
   \advance \grcalca by \grcalcd
   \grcalcc = \factor
   \advance \grcalcc by \hfactor
   \grcalcd = #1
   \advance \grcalcd by -1
   \multiply \grcalcd by \factor
   \put(\grcalca,\grcalcb){\oval(\grcalcd,\grcalcc)[t]}
   \advance \grcalcb by \factor
   \put(\grcalca,\grcalcb) {\line(0,-1){\qfactor}} 
   \advance \grcolumn by #1}

 \newcommand{\gev}{  
   \grcalca = \grcolumn
   \advance \grcalca by 1
   \multiply \grcalca by \factor
   \grcalcb = \grrow
   \multiply \grcalcb by \factor
   \grcalcc = \factor
   \advance \grcalcc by \hfactor
   \put(\grcalca,\grcalcb){\oval(\factor,\grcalcc)[b]}
   \advance \grcolumn by 2}

 \newcommand{\gdb}{   
   \grcalca = \grcolumn
   \advance \grcalca by 1
   \multiply \grcalca by \factor
   \grcalcb = \grrow
   \advance \grcalcb by -1
   \multiply \grcalcb by \factor
   \grcalcc = \factor
   \advance \grcalcc by \hfactor
   \put(\grcalca,\grcalcb){\oval(\factor,\grcalcc)[t]}
   \advance \grcolumn by 2}

 \newcommand{\gwev}[1]{    
   \grcalca = \grcolumn
   \multiply \grcalca by \factor
   \grcalcd = \hfactor
   \multiply \grcalcd by #1
   \advance \grcalca by \grcalcd
   \grcalcb = \grrow
   \multiply \grcalcb by \factor
   \grcalcc = \factor
   \advance \grcalcc by \hfactor
   \grcalcd = #1
   \advance \grcalcd by -1
   \multiply \grcalcd by \factor
   \put(\grcalca,\grcalcb){\oval(\grcalcd,\grcalcc)[b]}
   \advance \grcolumn by #1}

 \newcommand{\gwdb}[1]{   
   \grcalca = \grcolumn
   \multiply \grcalca by \factor
   \grcalcd = \hfactor
   \multiply \grcalcd by #1
   \advance \grcalca by \grcalcd
   \grcalcb = \grrow
   \advance \grcalcb by -1
   \multiply \grcalcb by \factor
   \grcalcc = \factor
   \advance \grcalcc by \hfactor
   \grcalcd = #1
   \advance \grcalcd by -1
   \multiply \grcalcd by \factor
   \put(\grcalca,\grcalcb){\oval(\grcalcd,\grcalcc)[t]}
   \advance \grcolumn by #1}

 \newcommand{\gbr}{
   \grcalca = \grcolumn
   \multiply \grcalca by \factor
   \advance \grcalca by \hfactor
   \grcalcb = \grcalca
   \advance \grcalcb by \hfactor
   \grcalcc = \grcalca
   \advance \grcalcc by \factor
   \grcalcd = \grrow
   \multiply \grcalcd by \factor
   \grcalce = \grcalcd
   \advance \grcalce by -\tfactor
   \grcalcf = \grcalcd
   \advance \grcalcf by -\hfactor
   \grcalcg = \grcalce
   \advance \grcalcg by -\tfactor
   \grcalch = \grcalcd
   \advance \grcalch by -\factor
   \qbezier(\grcalca,\grcalcd)(\grcalca,\grcalce)(\grcalcb,\grcalcf) 
   \qbezier(\grcalcb,\grcalcf)(\grcalcc,\grcalcg)(\grcalcc,\grcalch) 
   \advance \grcalcf by -\dfactor
   \advance \grcalcb by -\sfactor
   \qbezier(\grcalca,\grcalch)(\grcalca,\grcalcg)(\grcalcb,\grcalcf) 
   \advance \grcalcf by \sfactor
   \advance \grcalcb by \tfactor
   \qbezier(\grcalcc,\grcalcd)(\grcalcc,\grcalce)(\grcalcb,\grcalcf) 
   \advance \grcolumn by 2}

 \newcommand{\gibr}{
   \grcalca = \grcolumn
   \multiply \grcalca by \factor
   \advance \grcalca by \hfactor
   \grcalcb = \grcalca
   \advance \grcalcb by \hfactor
   \grcalcc = \grcalca
   \advance \grcalcc by \factor
   \grcalcd = \grrow
   \multiply \grcalcd by \factor
   \grcalce = \grcalcd
   \advance \grcalce by -\tfactor
   \grcalcf = \grcalcd
   \advance \grcalcf by -\hfactor
   \grcalcg = \grcalce
   \advance \grcalcg by -\tfactor
   \grcalch = \grcalcd
   \advance \grcalch by -\factor
   \qbezier(\grcalcc,\grcalcd)(\grcalcc,\grcalce)(\grcalcb,\grcalcf) 
   \qbezier(\grcalcb,\grcalcf)(\grcalca,\grcalcg)(\grcalca,\grcalch) 
   \advance \grcalcf by -\dfactor
   \advance \grcalcb by \sfactor
   \qbezier(\grcalcc,\grcalch)(\grcalcc,\grcalcg)(\grcalcb,\grcalcf) 
   \advance \grcalcf by \sfactor
   \advance \grcalcb by -\tfactor
   \qbezier(\grcalca,\grcalcd)(\grcalca,\grcalce)(\grcalcb,\grcalcf) 
   \advance \grcolumn by 2}

\newcommand{\gsy}{
   \grcalca = \grcolumn
   \multiply \grcalca by \factor
   \advance \grcalca by \hfactor
   \grcalcb = \grcalca
   \advance \grcalcb by \hfactor
   \grcalcc = \grcalca
   \advance \grcalcc by \factor
   \grcalcd = \grrow
   \multiply \grcalcd by \factor
   \grcalce = \grcalcd
   \advance \grcalce by -\tfactor
   \grcalcf = \grcalcd
   \advance \grcalcf by -\hfactor
   \grcalcg = \grcalce
   \advance \grcalcg by -\tfactor
   \grcalch = \grcalcd
   \advance \grcalch by -\factor
   \qbezier(\grcalcc,\grcalcd)(\grcalcc,\grcalce)(\grcalcb,\grcalcf) 
   \qbezier(\grcalcb,\grcalcf)(\grcalca,\grcalcg)(\grcalca,\grcalch) 
   \advance \grcalcf by -\dfactor
   \advance \grcalcb by \sfactor
   \qbezier(\grcalcc,\grcalch)(\grcalcc,\grcalcg)(\grcalcb,\grcalcf) 
   \qbezier(\grcalca,\grcalcd)(\grcalca,\grcalce)(\grcalcb,\grcalcf) 
   \advance \grcolumn by 2}

 \newcommand{\gbrc}{
   \grcalca = \grcolumn
   \multiply \grcalca by \factor
   \advance \grcalca by \hfactor
   \grcalcb = \grcalca
   \advance \grcalcb by \hfactor
   \grcalcc = \grcalca
   \advance \grcalcc by \factor
   \grcalcd = \grrow
   \multiply \grcalcd by \factor
   \grcalce = \grcalcd
   \advance \grcalce by -\tfactor
   \grcalcf = \grcalcd
   \advance \grcalcf by -\hfactor
   \grcalcg = \grcalce
   \advance \grcalcg by -\tfactor
   \grcalch = \grcalcd
   \advance \grcalch by -\factor
   \put(\grcalcb,\grcalcf){\circle{\hfactor}}
   \qbezier(\grcalca,\grcalcd)(\grcalca,\grcalce)(\grcalcb,\grcalcf) 
   \qbezier(\grcalcb,\grcalcf)(\grcalcc,\grcalcg)(\grcalcc,\grcalch) 
   \advance \grcalcf by -\dfactor
   \advance \grcalcb by -\sfactor
   \qbezier(\grcalca,\grcalch)(\grcalca,\grcalcg)(\grcalcb,\grcalcf) 
   \advance \grcalcf by \sfactor
   \advance \grcalcb by \tfactor
   \qbezier(\grcalcc,\grcalcd)(\grcalcc,\grcalce)(\grcalcb,\grcalcf) 
   \advance \grcolumn by 2}

 \newcommand{\gibrc}{
   \grcalca = \grcolumn
   \multiply \grcalca by \factor
   \advance \grcalca by \hfactor
   \grcalcb = \grcalca
   \advance \grcalcb by \hfactor
   \grcalcc = \grcalca
   \advance \grcalcc by \factor
   \grcalcd = \grrow
   \multiply \grcalcd by \factor
   \grcalce = \grcalcd
   \advance \grcalce by -\tfactor
   \grcalcf = \grcalcd
   \advance \grcalcf by -\hfactor
   \grcalcg = \grcalce
   \advance \grcalcg by -\tfactor
   \grcalch = \grcalcd
   \advance \grcalch by -\factor
   \put(\grcalcb,\grcalcf){\circle{\hfactor}}
   \qbezier(\grcalcc,\grcalcd)(\grcalcc,\grcalce)(\grcalcb,\grcalcf) 
   \qbezier(\grcalcb,\grcalcf)(\grcalca,\grcalcg)(\grcalca,\grcalch) 
   \advance \grcalcf by -\dfactor
   \advance \grcalcb by \sfactor
   \qbezier(\grcalcc,\grcalch)(\grcalcc,\grcalcg)(\grcalcb,\grcalcf) 
   \advance \grcalcf by \sfactor
   \advance \grcalcb by -\tfactor
   \qbezier(\grcalca,\grcalcd)(\grcalca,\grcalce)(\grcalcb,\grcalcf) 
   \advance \grcolumn by 2}

 \newcommand{\gu}[1]{
   \grcalca = \grcolumn
   \multiply \grcalca by \factor
   \grcalcd = \hfactor
   \multiply \grcalcd by #1
   \advance \grcalca by \grcalcd
   \grcalcb = \grrow
   \advance \grcalcb by -1
   \multiply \grcalcb by \factor
   \put(\grcalca,\grcalcb) {\line(0,1){\hfactor}} 
   \advance \grcalcb by \hfactor
   \put(\grcalca,\grcalcb) {\circle*{3}}
   \advance \grcolumn by #1}

 \newcommand{\gcu}[1]{
   \grcalca = \grcolumn
   \multiply \grcalca by \factor
   \grcalcd = \hfactor
   \multiply \grcalcd by #1
   \advance \grcalca by \grcalcd
   \grcalcb = \grrow
   \multiply \grcalcb by \factor
   \put(\grcalca,\grcalcb) {\line(0,-1){\hfactor}} 
   \advance \grcalcb by -\hfactor
   \put(\grcalca,\grcalcb) {\circle*{3}}
   \advance \grcolumn by #1}

 \newcommand{\gmp}[1]{
   \grcalca = \grcolumn
   \multiply \grcalca by \factor
   \advance \grcalca by \hfactor
   \grcalcb = \grrow
   \multiply \grcalcb by \factor
   \put(\grcalca,\grcalcb) {\line(0,-1){\dfactor}} 
   \advance \grcalcb by -\factor
   \put(\grcalca,\grcalcb) {\line(0,1){\dfactor}} 
   \advance \grcalcb by \hfactor
   \grcalcc = \factor
   \advance \grcalcc by -\qfactor
   \put(\grcalca,\grcalcb) {\circle{\grcalcc}}
   \put(\grcalca,\grcalcb) {\makebox(0,0){$\scriptstyle #1$}}
   \advance \grcolumn by 1}

 \newcommand{\gbmp}[1]{
   \grcalca = \grcolumn
   \multiply \grcalca by \factor
   \advance \grcalca by \hfactor
   \grcalcb = \grrow
   \multiply \grcalcb by \factor
   \put(\grcalca,\grcalcb) {\line(0,-1){\dfactor}} 
   \advance \grcalcb by -\factor
   \put(\grcalca,\grcalcb) {\line(0,1){\dfactor}} 
   \advance \grcalca by -\hfactor
   \advance \grcalca by \dfactor
   \advance \grcalcb by \dfactor
   \grcalcc = \factor
   \advance \grcalcc by -\sfactor
   \put(\grcalca,\grcalcb) {\framebox(\grcalcc,\grcalcc){$\scriptstyle #1$}}
   \advance \grcolumn by 1}

 \newcommand{\gbmpt}[1]{
   \grcalca = \grcolumn
   \multiply \grcalca by \factor
   \advance \grcalca by \hfactor
   \grcalcb = \grrow
   \multiply \grcalcb by \factor
   \put(\grcalca,\grcalcb) {\line(0,-1){\dfactor}} 
   \advance \grcalcb by -\factor
   \advance \grcalca by -\hfactor
   \advance \grcalca by \dfactor
   \advance \grcalcb by \dfactor
   \grcalcc = \factor
   \advance \grcalcc by -\sfactor
   \put(\grcalca,\grcalcb) {\framebox(\grcalcc,\grcalcc){$\scriptstyle #1$}}
   \advance \grcolumn by 1}

 \newcommand{\gbmpb}[1]{
   \grcalca = \grcolumn
   \multiply \grcalca by \factor
   \advance \grcalca by \hfactor
   \grcalcb = \grrow
   \multiply \grcalcb by \factor
   \advance \grcalcb by -\factor
   \put(\grcalca,\grcalcb) {\line(0,1){\dfactor}} 
   \advance \grcalca by -\hfactor
   \advance \grcalca by \dfactor
   \advance \grcalcb by \dfactor
   \grcalcc = \factor
   \advance \grcalcc by -\sfactor
   \put(\grcalca,\grcalcb) {\framebox(\grcalcc,\grcalcc){$\scriptstyle #1$}}
   \advance \grcolumn by 1}

 \newcommand{\gbmpn}[1]{
   \grcalca = \grcolumn
   \multiply \grcalca by \factor
   \advance \grcalca by \hfactor
   \grcalcb = \grrow
   \multiply \grcalcb by \factor
   \advance \grcalcb by -\factor
   \advance \grcalca by -\hfactor
   \advance \grcalca by \dfactor
   \advance \grcalcb by \dfactor
   \grcalcc = \factor
   \advance \grcalcc by -\sfactor
   \put(\grcalca,\grcalcb) {\framebox(\grcalcc,\grcalcc){$\scriptstyle #1$}}
   \advance \grcolumn by 1}

 \newcommand{\glmptb}{    
   \grcalca = \grcolumn
   \multiply \grcalca by \factor
   \advance \grcalca by \hfactor
   \grcalcb = \grrow
   \multiply \grcalcb by \factor
   \put(\grcalca,\grcalcb) {\line(0,-1){\dfactor}} 
   \advance \grcalcb by -\factor
   \put(\grcalca,\grcalcb) {\line(0,1){\dfactor}} 
   \advance \grcalca by -\hfactor
   \advance \grcalca by \dfactor
   \advance \grcalcb by \dfactor
   \put(\grcalca,\grcalcb) {\line(1,0){\factor}} 
   \advance \grcalcb by \factor
   \advance \grcalcb by -\sfactor
   \put(\grcalca,\grcalcb) {\line(1,0){\factor}} 
   \grcalcc = \factor
   \advance \grcalcc by -\sfactor
   \put(\grcalca,\grcalcb) {\line(0,-1){\grcalcc}} 
   \advance \grcolumn by 1}

 \newcommand{\glmpt}{    
   \grcalca = \grcolumn
   \multiply \grcalca by \factor
   \advance \grcalca by \hfactor
   \grcalcb = \grrow
   \multiply \grcalcb by \factor
   \put(\grcalca,\grcalcb) {\line(0,-1){\dfactor}} 
   \advance \grcalca by -\hfactor
   \advance \grcalca by \dfactor
   \advance \grcalcb by -\dfactor
   \put(\grcalca,\grcalcb) {\line(1,0){\factor}} 
   \advance \grcalcb by -\factor
   \advance \grcalcb by \sfactor
   \put(\grcalca,\grcalcb) {\line(1,0){\factor}} 
   \grcalcc = \factor
   \advance \grcalcc by -\sfactor
   \put(\grcalca,\grcalcb) {\line(0,1){\grcalcc}} 
   \advance \grcolumn by 1}

 \newcommand{\glmpb}{    
   \grcalca = \grcolumn
   \multiply \grcalca by \factor
   \advance \grcalca by \hfactor
   \grcalcb = \grrow
   \multiply \grcalcb by \factor
   \advance \grcalcb by -\factor
   \put(\grcalca,\grcalcb) {\line(0,1){\dfactor}} 
   \advance \grcalca by -\hfactor
   \advance \grcalca by \dfactor
   \advance \grcalcb by \dfactor
   \put(\grcalca,\grcalcb) {\line(1,0){\factor}} 
   \advance \grcalcb by \factor
   \advance \grcalcb by -\sfactor
   \put(\grcalca,\grcalcb) {\line(1,0){\factor}} 
   \grcalcc = \factor
   \advance \grcalcc by -\sfactor
   \put(\grcalca,\grcalcb) {\line(0,-1){\grcalcc}} 
   \advance \grcolumn by 1}

 \newcommand{\glmp}{    
   \grcalca = \grcolumn
   \multiply \grcalca by \factor
   \advance \grcalca by \dfactor
   \grcalcb = \grrow
   \multiply \grcalcb by \factor
   \advance \grcalcb by -\dfactor
   \put(\grcalca,\grcalcb) {\line(1,0){\factor}} 
   \advance \grcalcb by -\factor
   \advance \grcalcb by \sfactor
   \put(\grcalca,\grcalcb) {\line(1,0){\factor}} 
   \grcalcc = \factor
   \advance \grcalcc by -\sfactor
   \put(\grcalca,\grcalcb) {\line(0,1){\grcalcc}} 
   \advance \grcolumn by 1}

 \newcommand{\gcmptb}{    
   \grcalca = \grcolumn
   \multiply \grcalca by \factor
   \advance \grcalca by \hfactor
   \grcalcb = \grrow
   \multiply \grcalcb by \factor
   \put(\grcalca,\grcalcb) {\line(0,-1){\dfactor}} 
   \advance \grcalcb by -\factor
   \put(\grcalca,\grcalcb) {\line(0,1){\dfactor}} 
   \advance \grcalca by -\hfactor
   \advance \grcalcb by \dfactor
   \put(\grcalca,\grcalcb) {\line(1,0){\factor}} 
   \advance \grcalcb by \factor
   \advance \grcalcb by -\sfactor
   \put(\grcalca,\grcalcb) {\line(1,0){\factor}} 
   \advance \grcolumn by 1}

\newcommand{\gmpcu}[1]{
   \grcalca = \grcolumn
   \multiply \grcalca by \factor
   \advance \grcalca by \hfactor
   \grcalcb = \grrow
   \multiply \grcalcb by \factor
   \put(\grcalca,\grcalcb) {\line(0,-1){\dfactor}} 
   \advance \grcalcb by -\factor
   \advance \grcalcb by \hfactor
   \grcalcc = \factor
   \advance \grcalcc by -\qfactor
   \put(\grcalca,\grcalcb) {\circle{\grcalcc}}
   \put(\grcalca,\grcalcb) {\makebox(0,0){$\scriptstyle #1$}}
   \advance \grcolumn by 1}
   
\newcommand{\gmpu}[1]{
   \grcalca = \grcolumn
   \multiply \grcalca by \factor
   \advance \grcalca by \hfactor
   \grcalcb = \grrow
   \multiply \grcalcb by \factor
   \advance \grcalcb by -\factor
   \put(\grcalca,\grcalcb) {\line(0,1){\dfactor}} 
   \advance \grcalcb by \hfactor
   \grcalcc = \factor
   \advance \grcalcc by -\qfactor
   \put(\grcalca,\grcalcb) {\circle{\grcalcc}}
   \put(\grcalca,\grcalcb) {\makebox(0,0){$\scriptstyle #1$}}
   \advance \grcolumn by 1}      

 \newcommand{\gcmpt}{    
   \grcalca = \grcolumn
   \multiply \grcalca by \factor
   \advance \grcalca by \hfactor
   \grcalcb = \grrow
   \multiply \grcalcb by \factor
   \put(\grcalca,\grcalcb) {\line(0,-1){\dfactor}} 
   \advance \grcalcb by -\factor
   \advance \grcalca by -\hfactor
   \advance \grcalcb by \dfactor
   \put(\grcalca,\grcalcb) {\line(1,0){\factor}} 
   \advance \grcalcb by \factor
   \advance \grcalcb by -\sfactor
   \put(\grcalca,\grcalcb) {\line(1,0){\factor}} 
   \advance \grcolumn by 1}

 \newcommand{\gcmpb}{    
   \grcalca = \grcolumn
   \multiply \grcalca by \factor
   \advance \grcalca by \hfactor
   \grcalcb = \grrow
   \multiply \grcalcb by \factor
   \advance \grcalcb by -\factor
   \put(\grcalca,\grcalcb) {\line(0,1){\dfactor}} 
   \advance \grcalca by -\hfactor
   \advance \grcalcb by \dfactor
   \put(\grcalca,\grcalcb) {\line(1,0){\factor}} 
   \advance \grcalcb by \factor
   \advance \grcalcb by -\sfactor
   \put(\grcalca,\grcalcb) {\line(1,0){\factor}} 
   \advance \grcolumn by 1}

 \newcommand{\gcmp}{    
   \grcalca = \grcolumn
   \multiply \grcalca by \factor
   \grcalcb = \grrow
   \multiply \grcalcb by \factor
   \advance \grcalcb by -\factor
   \advance \grcalcb by \dfactor
   \put(\grcalca,\grcalcb) {\line(1,0){\factor}} 
   \advance \grcalcb by \factor
   \advance \grcalcb by -\sfactor
   \put(\grcalca,\grcalcb) {\line(1,0){\factor}} 
   \advance \grcolumn by 1}

 \newcommand{\grmptb}{    
   \grcalca = \grcolumn
   \multiply \grcalca by \factor
   \advance \grcalca by \hfactor
   \grcalcb = \grrow
   \multiply \grcalcb by \factor
   \put(\grcalca,\grcalcb) {\line(0,-1){\dfactor}} 
   \advance \grcalcb by -\factor
   \put(\grcalca,\grcalcb) {\line(0,1){\dfactor}} 
   \advance \grcalca by \hfactor
   \advance \grcalca by -\dfactor
   \advance \grcalcb by \dfactor
   \put(\grcalca,\grcalcb) {\line(-1,0){\factor}} 
   \advance \grcalcb by \factor
   \advance \grcalcb by -\sfactor
   \put(\grcalca,\grcalcb) {\line(-1,0){\factor}} 
   \grcalcc = \factor
   \advance \grcalcc by -\sfactor
   \put(\grcalca,\grcalcb) {\line(0,-1){\grcalcc}} 
   \advance \grcolumn by 1}

 \newcommand{\grmpt}{    
   \grcalca = \grcolumn
   \multiply \grcalca by \factor
   \advance \grcalca by \hfactor
   \grcalcb = \grrow
   \multiply \grcalcb by \factor
   \put(\grcalca,\grcalcb) {\line(0,-1){\dfactor}} 
   \advance \grcalca by \hfactor
   \advance \grcalca by -\dfactor
   \advance \grcalcb by -\dfactor
   \put(\grcalca,\grcalcb) {\line(-1,0){\factor}} 
   \advance \grcalcb by -\factor
   \advance \grcalcb by \sfactor
   \put(\grcalca,\grcalcb) {\line(-1,0){\factor}} 
   \grcalcc = \factor
   \advance \grcalcc by -\sfactor
   \put(\grcalca,\grcalcb) {\line(0,1){\grcalcc}} 
   \advance \grcolumn by 1}

 \newcommand{\grmpb}{    
   \grcalca = \grcolumn
   \multiply \grcalca by \factor
   \advance \grcalca by \hfactor
   \grcalcb = \grrow
   \multiply \grcalcb by \factor
   \advance \grcalcb by -\factor
   \put(\grcalca,\grcalcb) {\line(0,1){\dfactor}} 
   \advance \grcalca by \hfactor
   \advance \grcalca by -\dfactor
   \advance \grcalcb by \dfactor
   \put(\grcalca,\grcalcb) {\line(-1,0){\factor}} 
   \advance \grcalcb by \factor
   \advance \grcalcb by -\sfactor
   \put(\grcalca,\grcalcb) {\line(-1,0){\factor}} 
   \grcalcc = \factor
   \advance \grcalcc by -\sfactor
   \put(\grcalca,\grcalcb) {\line(0,-1){\grcalcc}} 
   \advance \grcolumn by 1}

 \newcommand{\grmp}{    
   \grcalca = \grcolumn
   \multiply \grcalca by \factor
   \advance \grcalca by \factor
   \advance \grcalca by -\dfactor
   \grcalcb = \grrow
   \multiply \grcalcb by \factor
   \advance \grcalcb by -\dfactor
   \put(\grcalca,\grcalcb) {\line(-1,0){\factor}} 
   \advance \grcalcb by -\factor
   \advance \grcalcb by \sfactor
   \put(\grcalca,\grcalcb) {\line(-1,0){\factor}} 
   \grcalcc = \factor
   \advance \grcalcc by -\sfactor
   \put(\grcalca,\grcalcb) {\line(0,1){\grcalcc}} 
   \advance \grcolumn by 1}

 \newcommand{\gwmuh}[3]{    
   \grcalca = \grcolumn
   \multiply \grcalca by \factor
   \grcalcb = #2
   \advance \grcalcb by #3
   \multiply \grcalcb by \qfactor
   \advance \grcalca by \grcalcb
   \grcalcb = \grrow
   \multiply \grcalcb by \factor
   \grcalcc = #3
   \advance \grcalcc by -#2
   \multiply \grcalcc by \hfactor
   \grcalcd = \factor
   \advance \grcalcd by \hfactor
   \put(\grcalca,\grcalcb){\oval(\grcalcc,\grcalcd)[b]}
   \grcalca = \grcolumn
   \multiply \grcalca by \factor
   \grcalcc = #1
   \multiply \grcalcc by \hfactor
   \advance \grcalca by \grcalcc
   \advance \grcalcb by -\hfactor
   \advance \grcalcb by -\qfactor
   \put(\grcalca,\grcalcb) {\line(0,-1){\qfactor}} 
   \advance \grcolumn by #1}

 \newcommand{\gwcmh}[3]{   
   \grcalca = \grcolumn
   \multiply \grcalca by \factor
   \grcalcb = #2
   \advance \grcalcb by #3
   \multiply \grcalcb by \qfactor
   \advance \grcalca by \grcalcb
   \grcalcb = \grrow
   \advance \grcalcb by -1
   \multiply \grcalcb by \factor
   \grcalcc = #3
   \advance \grcalcc by -#2
   \multiply \grcalcc by \hfactor
   \grcalcd = \factor
   \advance \grcalcd by \hfactor
   \put(\grcalca,\grcalcb){\oval(\grcalcc,\grcalcd)[t]}
   \grcalca = \grcolumn
   \multiply \grcalca by \factor
   \grcalcc = #1
   \multiply \grcalcc by \hfactor
   \advance \grcalca by \grcalcc
   \advance \grcalcb by \factor
   \put(\grcalca,\grcalcb) {\line(0,-1){\qfactor}} 
   \advance \grcolumn by #1}

 \newcommand{\gsbox}[1]{
   \grcalca = \grcolumn
   \multiply \grcalca by \factor
   \grcalcb = \grrow
   \multiply \grcalcb by \factor
   \advance \grcalcb by -\factor
   \grcalcc = #1
   \multiply \grcalcc by \factor
   \grcalcd = \factor
   \put(\grcalca,\grcalcb){\framebox(\grcalcc,\grcalcd){}}}

\allowdisplaybreaks[4]

\begin{document}
\title[Quasi-Hopf algebras of dimension 6]
{Quasi-Hopf algebras of dimension $6$}
\author{D. Bulacu}
\address{Faculty of Mathematics and Informatics, University
of Bucharest, Str. Academiei 14, RO-010014 Bucharest 1, Romania}
\email{daniel.bulacu@fmi.unibuc.ro}
\author{M. Misurati}
\address{University of Ferrara, Department of Mathematics, Via Machiavelli
35, Ferrara, I-44121, Italy}
\email{matteo.misurati@unife.it}
\thanks{
The first author was supported by the grant PCE47-2022 of UEFISCDI,
project PN-III-P4-PCE-2021-0282; he also thanks the University of Ferrara (Italy) for its warm hospitality.} 

\begin{abstract}
We complete the classification of the $6$-dimensional quasi-Hopf algebras, by proving that any such algebra is semisimple. As byproducts, we provide 
examples of $6$-dimensional quasi-bialgebras that are not semisimple as algebras, as well as the concrete quasi-Hopf structures of the 
$6$-dimensional semisimple quasi-Hopf algebras previously classified by Etingof and Gelaki in terms of their category of representations. In total there are $15$ quasi-Hopf algebras in dimension $6$ which are not pairwise twist equivalent.   
\end{abstract}
\maketitle
\section*{Introduction}\selabel{intro}
\setcounter{equation}{0}
The classification of quasi-Hopf algebras is in an early age, with most advancements focusing on the semisimple or basic case. 
On the one hand, this is own to the fact that the semisimple quasi-Hopf algebras characterize fusion categories with integer 
Frobenius-Perron dimensions of simple objects. Thus the classification of such fusion categories provides us for free the classification 
of the semisimple quasi-Hopf algebras and vice versa. For instance, this is the case in dimension $p$ and $pq$, see \cite[Corollary 8.31]{eno} 
and \cite[Theorem 6.3]{ego}, respectively, provided that $p<q$ are prime positive integers. More generally, owing to \cite[Proposition 2.6]{eo}, 
quasi-Hopf algebras are given, through the Tannaka-Krein reconstruction procedure, by the finite tensor categories 
with integer Frobenius-Perron dimensions of objects. In general, the problem of classifying fusion categories or finite tensor categories is an  
extremely difficult task, and so the same is the classification of the quasi-Hopf algebras in a given dimension. 
On the other hand, parallel to the pointed case in Hopf algebra theory and motivated by the categorical results mentioned above, there is a high interest also in 
the classification of the so called basic quasi-Hopf algebras (quasi-Hopf algebras with all irreducible representations $1$-dimensional). 
The first step in this direction was made by Etingof and Gelaky in \cite{eg3}, where are classified the finite dimensional quasi-Hopf algebras (over a field 
of characteristic zero and algebraically closed) which have two 1-dimensional irreducible representations; we must note that in \cite{eg3}, as a byproduct, 
the authors also obtain the classification of the $4$-dimensional quasi-Hopf algebras. Last but not least, a 
class of basic quasi-Hopf algebras of dimension $n^3$ is constructed by Gelaki in \cite{g5}, and \cite{eg4} uncovers    
the classification of the radically graded finite dimensional quasi-Hopf algebra with Jacobson radical of prime codimension. 
 
The aim of this paper is to complete the classification of the $6$-dimensional quasi-Hopf algebras. In the Hopf algebra case, in dimension $6$ the 
classification of semisimple Hopf algebras was given by Masuoka in \cite{akira}. Afterwards, \c Stefan \cite{dspq} showed that any Hopf algebra of dimension $6$ 
is semisimple, and thus finished the classification in dimension $6$. For quasi-Hopf algebras, owing to the classifications we have for fusion 
categories with integer Frobenius-Perron dimensions of simple objects, we know the classification of the semisimple quasi-Hopf algebras in dimension $6$. 
So, as \c Stefan did in the Hopf case, we complete the classification of quasi-Hopf algebras in dimension $6$ by proving that any such algebra is semisimple. 
To this end, we have to follow a path different from the one used in the Hopf case, because for a quasi-Hopf algebra $H$ we cannot consider the coradical filtration 
($H$ is not a coassociative coalgebra) and so there is no chance of producing a Taft-Wilson theorem for quasi-Hopf algebras that generalizes the one for 
Hopf algebras; see \cite[Theorem 7.7.7]{DNRha} or \cite[Theorem 5.4.1]{mont}. 
Fortunately, an equally useful tool that can be used in the quasi-Hopf setting is the quasi-Hopf algebra filtration on $H$ defined by the powers of the 
radical Jacobson of $H$; note that, in the Hopf case, the radical and coradical filtration are related as in \cite[Proposition 5.2.9]{mont}. 
The algebra associated to this filtration, denoted by ${\rm gr}(H)$, is a graded quasi-Hopf algebra with projection which possesses 
enough properties to be able to classify basic quasi-Hopf algebras, at least in small dimensions. For instance, in our case, by using biproducts in the sense 
of \cite{bn} and making use of arguments similar to those used by Masuoka in \cite{akira}, ${\rm gr}(H)$ helps us to show that any quasi-Hopf algebra of dimension $6$ 
is semisimple. Consequently, we complete in this way the classification in dimension $6$ both for quasi-Hopf algebras and finite tensor categories 
with integer Frobenius-Perron dimensions of objects.    
 
The paper is structured as follows. In \seref{prelim} we recall the basic things related to a quasi-Hopf algebra; for more information about this topic 
we invite the reader to consult \cite{bcpvo, kas}. For $H$ a $6$-dimensional basic non-semisimple quasi-Hopf algebra, $A:={\rm gr}(H)$ might be a biproduct 
quasi-Hopf algebra between a braided Hopf of dimension $2$ or $3$ and a group quasi-Hopf algebra (with reassociator determined by a $3$-cocycle of the group). 
In dimension $2$ these braided Hopf algebras were classified in \cite{bm}. In Section \ref{3dH} we classify those of dimension $3$ within the category of Yetter-Drinfeld modules over $k[C_2]$, viewed as a quasi-Hopf algebra via the unique non-trivial $3$-cocycle of $C_2$; in general, we denote by $C_n$ the cyclic group of order $n$. Since in both cases the biproduct quasi-Hopf algebras that we obtain are semisimple quasi-Hopf algebras, we conclude in \seref{6dmqHas} that there 
are no non-semisimple quasi-Hopf algebras of dimension $6$. We should stress the fact that our classification results from Section \ref{3dH} 
give rise to a $3$ non-isomorphic quasi-bialgebras of dimension $6$ that are not semisimple (and so they are not quasi-Hopf algebras). Finally, in \thref{fulldes6dimqHas} we give the explicit structures of the quasi-Hopf algebras of dimension $6$ which are not pairwise twist equivalent; the classification result 
in \thref{fulldes6dimqHas} is mostly obtained from \cite[Theorem 6.3]{ego}.   

\section{Preliminaries}\selabel{prelim}
\setcounter{equation}{0}
\setcounter{equation}{0}
\subsection{Quasi-bialgebras and quasi-Hopf algebras.}
Throughout this paper, $k$ is a field of characteristic zero and algebraically closed. 
All algebras, linear spaces, etc. will be over $k$; unadorned $\ot $ means $\ot_k$.
Following Drinfeld \cite{dri}, a quasi-bialgebra is
a quadruple $(H, \Delta , \va , \Phi )$ where $H$ is
an associative algebra with unit $1$, 
$\Phi$ is an invertible element in $H\ot H\ot H$, and
$\Delta :\ H\ra H\ot H$ and $\va :\ H\ra k$ are algebra
homomorphisms obeying the identities
\begin{eqnarray}
&&(\Id_H \ot \Delta )(\Delta (h))=
\Phi (\Delta \ot \Id_H)(\Delta (h))\Phi ^{-1},\eqlabel{q1}\\
&&(\Id_H \ot \va )(\Delta (h))=h~~,~~
(\va \ot \Id_H)(\Delta (h))=h,\eqlabel{q2}
\end{eqnarray}
for all $h\in H$, where
$\Phi$ is a normalized $3$-cocycle, in the sense that
\begin{eqnarray}
&&(1\ot \Phi)(\Id_H\ot \Delta \ot \Id_H)
(\Phi)(\Phi \ot 1)\nonumber\\
&&\hspace*{1.5cm}
=(\Id_H\ot \Id_H \ot \Delta )(\Phi )
(\Delta \ot \Id_H \ot \Id_H)(\Phi),\eqlabel{q3}\\
&&(\Id \ot \va \ot \Id_H)(\Phi)=1\ot 1.\eqlabel{q4}
\end{eqnarray}
The map $\Delta$ is called the coproduct or the
comultiplication, $\va $ is the counit, and $\Phi $ is the
reassociator. As for Hopf algebras, we denote $\Delta (h)=h_1\ot h_2$,
but since $\Delta $ is only quasi-coassociative we adopt the
further convention (summation understood):
\[
(\Delta \ot \Id_H)(\Delta (h))=h_{11}\ot h_{12}\ot h_2~~,~~
(\Id_H\ot \Delta)(\Delta (h))=h_1\ot h_{21}\ot h_{22},
\]
for all $h\in H$. We will denote the tensor components of $\Phi$
by capital letters, and the ones of $\Phi^{-1}$ by lower case letters, namely
\begin{eqnarray*}
&&\Phi=X^1\ot X^2\ot X^3=T^1\ot T^2\ot T^3=
V^1\ot V^2\ot V^3=\cdots\\
&&\Phi^{-1}=x^1\ot x^2\ot x^3=t^1\ot t^2\ot t^3=
v^1\ot v^2\ot v^3=\cdots
\end{eqnarray*}
$H$ is called a quasi-Hopf
algebra if, moreover, there exists an
anti-morphism $S$ of the algebra
$H$ and elements $\a , \b \in
H$ such that, for all $h\in H$, we
have:
\begin{eqnarray}
&&
S(h_1)\a h_2=\va(h)\a
~~{\rm and}~~
h_1\b S(h_2)=\va (h)\b,\eqlabel{q5}\\ 
&&X^1\b S(X^2)\a X^3=1
~~{\rm and}~~
S(x^1)\a x^2\b S(x^3)=1.\eqlabel{q6}
\end{eqnarray}

Our definition of a quasi-Hopf algebra is different from the
one given by Drinfeld \cite{dri} in the sense that we do not
require the antipode to be bijective. In the case where $H$ is finite dimensional
or quasi-triangular, bijectivity of the antipode follows from the other axioms,
see \cite{bc1} and \cite{bn3}, so the two definitions are equivalent. Anyway, the bijectivity 
of the antipode $S$ will be implicitly understood in the case when $S^{-1}$, the inverse 
of $S$, appears is formulas or computations.

Let $H$ be a quasi-bialgebra and $F\in H\otimes H$ an invertible element such that $\va(F^1)F^2=\va(F^2)F^1=1$, where 
$F=F^1\ot F^2$ is the formal notation for the tensor components of $F$; a similar notation we adopt for $F^{-1}$, $F^{-1}=G^1\ot G^2$. Note that 
$F$ is called a twist or gauge transformation for $H$. Let 
\begin{eqnarray*}
&&\Delta_F: H\ra H\ot H,~~\Delta_F(h)=F\Delta(h)F^{-1},\\
&&\Phi_F=(1_H\ot F)(\Id_H\ot \Delta)(F)\Phi (\Delta \ot \Id_H)(F^{-1})(F^{-1}\ot 1_H),
\end{eqnarray*}
Then $H_F:=(H, \Delta_F, \va , \Phi_F)$ is a quasi-bialgebra as well. Furthermore, if $H$ is a quasi-Hopf algebra then so is 
$H_F$ with $S_F=S$, $\a_F=S(G^1)\a G^2$ and $\b_F=F^1\b S(F^2)$.  

Two quasi-bialgebras (resp. quasi-Hopf algebras) $H$ and $H'$ are 
called twist equivalent if there exists a twist $F\in H'\ot H'$ such that $H$ and $H'_F$ are isomorphic 
as quasi-bialgebras (resp. quasi-Hopf algebras).

\subsection{Yetter-Drinfeld module categories and braided Hopf algebras}
The category of left Yetter-Drinfeld modules over a quasi-bialgebra $H$, denoted by ${}_H^H{\cal YD}$, was introduced by Majid in \cite{m1}. 
The objects of ${}_H^H{\cal YD}$ are left $H$-modules $M$ on which $H$ coacts from the left (we denote 
by $\lambda_M:\ M\to H\ot M,~~\lambda_M(m)=m_{[-1]}\ot m_{[0]}$ the left $H$-coaction on $M$) such that $\va(m_{[-1]})m_{[0]}=m$ and, 
for all $m\in M$,  
\begin{eqnarray}
&&\hspace*{-1.5cm}
X^1m_{[-1]}\ot (X^2\cd m_{[0]})_{[-1]}X^3
\ot (X^2\cd m_{[0]})_{[0]}\nonumber\\
&&\hspace*{1cm}=X^1(Y^1\cd m)_{[-1]_1}Y^2\ot X^2(Y^1\cd m)_{[-1]_2}Y^3
\ot X^3\cd (Y^1\cd m)_{[0]}.\label{y1}
\end{eqnarray}
It is compatible with the left $H$-module structure on $M$, in the sense that   
\begin{equation}\label{y3}
h_1m_{[-1]}\ot h_2\cd m_{[0]}=(h_1\cd m)_{[-1]}h_2\ot (h_1\cd m)_{[0]},~\forall~h\in H,~m\in M.
\end{equation}
When $H$ is a quasi-Hopf algebra with bijective antipode, ${}_H^H{\cal YD}$ is braided. 
The monoidal structure on ${}_H^H{\cal YD}$ is such that the forgetful functor ${}_H^H{\cal YD}\ra {}_H{\cal M}$ is strong monoidal. 
The coaction on the tensor product $M\ot N$ of two Yetter-Drinfeld modules $M$, $N$ is given, for all $m\in M$ and $n\in N$, by 
\begin{equation}\eqlabel{y4}
m\ot n\mapsto X^1(x^1Y^1\cdot m)_{[-1]} x^2(Y^2\cd n)_{[-1]}Y^3
\ot X^2\cd (x^1Y^1\cdot m)_{[0]}\ot X^3x^3\cdot (Y^2\cdot n)_{[0]}.
\end{equation}

In what follows, we call an algebra, coalgebra etc. in ${}_H^H{\cal YD}$ a Yetter-Drinfeld algebra, coalgebra etc. ($YD$-algebra, coalgebra etc. for short). 
Record that a YD-algebra is a Yetter-Drinfeld module $B$ equipped with a unital multiplication $\un{m}_B: B\ot B\ni b\ot b'\mapsto bb'\in B$, $1_B\in B$ is our   
notation for the unit of $B$, such that 
\begin{eqnarray}
&&h\cdot (\mf{b}\mf{b}')=(h_1\cdot \mf{b})(h_2\cdot \mf{b}')~,~h\cdot 1_B=\va(h)1_B~ ,~  {1_B}_{[-1]}\ot {1_B}_{[0]}=1_H\ot 1_B   \label{modalg1}\\
&&(\mf{b}\mf{b}')\mf{b}{''}=(X^1\cdot \mf{b})[(X^2\cdot \mf{b}')(X^3\cdot \mf{b}{''})],\label{modalg2}\\
&&(\mf{b}\mf{b}')_{[-1]}\ot (\mf{b}\mf{b}')_{[0]}\nonumber\\
&&\hspace{2cm}
=X^1(x^1Y^1\cdot \mf{b})_{[-1]} x^2(Y^2\cd \mf{b}')_{[-1]}Y^3
\ot (X^2\cd (x^1Y^1\cdot \mf{b})_{[0]})(X^3x^3\cdot (Y^2\cdot \mf{b}')_{[0]}),\label{modalg3}
\end{eqnarray}
for all $h\in H$ and $\mf{b}, \mf{b}', \mf{b}{''}\in B$. Note that (\ref{modalg1}) expresses the $H$-linearity of the multiplication and unit morphisms of $B$, 
(\ref{modalg2}) is just the associativity of $\un{m}_B$ in ${}_H{\cal M}$ and ${}_H^H{\cal YD}$ and (\ref{modalg3}) is the $H$-colinearity of $\un{m}_B$ 
in ${}_H^H{\cal YD}$; see \equref{y4}. Also, (\ref{modalg1}, \ref{modalg2}) are the required condition on $B$ to be an algebra in ${}_H{\cal M}$ or, equivalently, 
an $H$-module algebra.  

Likewise, a YD-coalgebra is a Yetter-Drinfeld module $B$ endowed with a comultiplication 
$\un{\Delta}_B: B\ni b\mapsto b_{\un{1}}\ot b_{\un{2}}\in B\ot B$ and counit $\un{\va}_B: B\ra k$ such that, for all $b\in B$,
\begin{eqnarray}
&&\un{\Delta}_B(h\cdot \mf{b})=h_1\cdot \mf{b}_{\un{1}}\ot h_2\cdot \mf{b}_{\un{2}},~~\un{\va}_B(h\cdot \mf{b})=\va(h)\un{\va}_B(\mf{b}),~~
\un{\va}_B(\mf{b}_{\un{1}})\mf{b}_{\un{2}}=\mf{b}=\un{\va}_B(\mf{b}_{\un{2}})\mf{b}_{\un{1}},\eqlabel{ydc1}\label{YDcoal con1}\\
&&X^1\cdot \mf{b}_{\un{1}\un{1}}\ot X^2\cdot \mf{b}_{\un{1}\un{2}}\ot X^3\cdot \mf{b}_{\un{2}}=
\mf{b}_{\un{1}}\ot \mf{b}_{\un{2}\un{1}}\ot \mf{b}_{\un{2}\un{2}},\eqlabel{ydc2}\\
&&\mf{b}_{[-1]}\ot \mf{b}_{[0]_{\un{1}}}\ot \mf{b}_{[0]_{\un{2}}}=X^1(x^1Y^1\cdot \mf{b}_{\un{1}})_{[-1]}x^2(Y^2\cdot \mf{b}_{\un{2}})_{[-1]}Y^3\nonumber\\
&&\hspace{2cm}\ot 
X^2\cdot (x^1Y^1\cdot \mf{b}_{\un{1}})_{[0]}\ot X^3x^3\cdot (Y^2\cdot \mf{b}_{\un{2}})_{[0]},~~
\un{\va}_B(\mf{b}_{[0]})\mf{b}_{[-1]}=\un{\va}_B(\mf{b})1.\eqlabel{ydc3}
\end{eqnarray} 
Observe that \equref{ydc1} expresses the fact that $\un{\Delta}_B$, $\un{\va}_B$ are left $H$-linear and $\un{\va}_B$ is counit for $\un{\Delta}_B$; otherwise stated, 
$B$ is a left $H$-module coalgebra, this means a coalgebra in ${}_H{\cal M}$. The relation \equref{ydc2} is the coassociativity of $\un{\Delta}_B$ in 
${}_H{\cal M}$ or, equivalently, in ${}_H^H{\cal YD}$ (since the forgetful functor from ${}_H^H{\cal YD}$ to ${}_H{\cal M}$ is strong monoidal), 
while \equref{ydc3} expresses the left $H$-colinearity of $\un{\Delta}_B$, $\un{\va}_B$ in ${}_H^H{\cal YD}$; see \equref{y4}. Also, we denoted 
$(\un{\Delta}_B\ot \Id_B)(\un{\Delta}_B(\mf{b})):=\mf{b}_{\un{1}\un{1}}\ot \mf{b}_{\un{1}\un{2}}\ot \mf{b}_{\un{2}}$ and 
$(\Id_B\ot \un{\Delta}_B)(\un{\Delta}_B(\mf{b})):=\mf{b}_{\un{1}}\ot \mf{b}_{\un{2}\un{1}}\ot \mf{b}_{\un{2}\un{2}}$. 

A bialgebra in ${}_H^H{\cal YD}$ is a YD-algebra $B$ that is at the same time a YD-coalgebra such that the comultiplication $\un{\Delta}_B$ and the counit 
$\un{\va}_B$ are algebra morphisms in ${}_H^H{\cal YD}$; $B\ot B$ is considered as an algebra in ${}_H^H{\cal YD}$ via the tensor product algebra structure. 
Explicitly, for all $b, b'\in B$, 
\begin{eqnarray}
&&\un{\Delta}_B(1_{B})=1_B\ot 1_B~,~\un{\va}_B(1_B)=1_k~,~\un{\va}_B(\mf{b}\mf{b}')=\un{\va}_B(\mf{b})\un{\va}_B(\mf{b}'),\label{moltcon1}\\
&&\un{\Delta}_B(\mf{b}\mf{b}')=(y^1X^1\hspace{-1mm}\cdot \mf{b}_{\un{1}})(y^2Y^1(x^1X^2\hspace{-1mm}\cdot \mf{b}_{\un{2}})_{[-1]}x^2X_1^3\hspace{-1mm}\cdot 
\mf{b}'_{\un{1}})\ot (y_1^3Y^2\cdot (x^1X^2\hspace{-1mm}\cdot \mf{b}_{\un{2}})_{[0]})(y_2^3Y^3x^3X_2^3\hspace{-1mm}\cdot \mf{b}'_{\un{2}}).\label{moltcon2}
\end{eqnarray}

A Hopf algebra in ${}_H^H{\cal YD}$ is a YD-bialgebra $B$ for which there exists a morphism $\un{S}_B: B\ra B$ in ${}_H^H{\cal YD}$, called antipode, such that 
$\un{S}_B(b_{\un{1}})b_{\un{2}}=\un{\va}_B(b)1_B=b_{\un{1}}\un{S}(b_{\un{2}})$, for all $b\in B$.
\subsection{Biproduct quasi-Hopf algebras}
Owing to \cite{db, panfred}, to a YD-algebra $B$ one can associate a $k$-algebra $B\# H$, called the smash product of $B$ and $H$ (it is suffices for  
$B$ to be an $H$-module algebra only). As a vector space $B\# H$ equals $B\ot H$, and its multiplication is defined by 
\begin{equation}\eqlabel{smashp}
    (\mf{b}\# h)(\mf{b}'\# h')=(x^1\cdot \mf{b})(x^2h_1\cdot \mf{b}')\# x^3h_2h',
\end{equation}
for all $\mf{b}, \mf{b}'\in B$ and $h, h'\in H$; the above multiplication is unital with unit $1_B\times 1$.

Similarly, owing to \cite{db}, to a $YD$-coalgebra $B$ one can associate an $H$-bimodule coalgebra; this means a coalgebra within the category of $H$-bimodules 
${}_H{\cal M}_H$, endowed with the monoidal structure coming from its identification to the category of left representations over 
the tensor product quasi-Hopf algebra $H\ot H^{\rm op}$ ($H^{\rm op}$ is the opposite quasi-Hopf algebra associated to $H$). This coalgebra 
structure, denoted by $B\tie H$ and called the smash product coalgebra of $B$ and $H$, is built on the $k$-vector space $B\ot H$ as follows 
(we write $\mf{b}\tie h$ in place of $\mf{b}\ot h$ to distinguish this coalgebra structure of $B\ot H$):
\begin{equation}\eqlabel{comultsmashprodcoalg}
\un{\Delta}(\mf{b}\tie h)=y^1X^1\cdot \mf{b}_{\un{1}}\tie y^2Y^1(x^1X^2\cdot \mf{b}_{\un{2}})_{[-1]}x^2X^3_1h_1\ot 
y^3_1Y^2\cdot (x^1X^2\cdot \mf{b}_{\un{2}})_{[0]}\tie y^3_2Y^3x^3X^3_2h_2
\end{equation}
and $\un{\va}(\mf{b}\tie h)=\un{\va}_B(\mf{b})\va(h)$, for all $\mf{b}\in B$, $h\in H$.

When $B$ is a YD-bialgebra, $B\# H$ and $B\tie H$ determine a quasi-bialgebra structure on $B\ot H$ with reassociator $1_B\ot X^1\ot 1_B\ot X^2\ot 1_B\ot X^3$,  
denoted in what follows by $B\times H$ and called the biproduct quasi-bialgebra of $B$ and $H$. Furthermore, 
$B\times H$ is a quasi-Hopf algebra with antipode 
\begin{equation}\eqlabel{antipbipr}
{\cal S}(\mf{b}\times h)=(1_B\times S(X^1x^1_1\mf{b}_{[-1]}h)\a)(X^2x^1_2\cdot \un{S}_B(\mf{b}_{[0]})\times X^3x^2\b S(x^3))
\end{equation}
and distinguished elements $1_B\times \a$ and $1_B\times \b$, provided that $B$ is a YD-Hopf algebra with antipode $\un{S}_B$; we refer to \cite{bn} for more details.   

According to \cite{db, bn}, biproduct quasi-Hopf algebras characterize the quasi-Hopf algebras with a projection: if there exist quasi-Hopf algebra morphisms 
$
\xymatrix{
H \ar[r]<2pt>^i &\ar[l]<2pt>^{\pi} A
}
$
such that $\pi i=\Id_H$ then there exists a Hopf algebra $B$ in ${}_H^H{\cal YD}$ such that $A$ is isomorphic to $B\times H$ as a quasi-Hopf algebra. 

\section{Free biproduct (quasi-)Hopf algebras of rank $3$ over $C_2$}\label{3dH}
\setcounter{equation}{0}
Recall that we are working over a field $k$ which is algebraically closed and of characteristic $0$. 
For $n$ a non-zero natural number, denote by $C_n$ the cyclic group of order $n$; unless otherwise stated, $g$ is a generator for $C_n$.   

Let $\htw$ be the group Hopf algebra $k[C_2]$, seen as a quasi-Hopf algebra with reassociator $\Phi_2=1\ot 1\ot 1 -2 p_-\ot p_-\ot p_- $, where 
$p_\pm=\frac{1}{2}(1\pm g)$. The antipode is the identity map and the distinguished elements are $\a=g$ and $\b=1$; in what follows, we denote this 
quasi-Hopf structure on $k[C_2]$ by $H(2)$.  
As a preliminary step in proving our main result, we describe all the $6$-dimensional quasi-Hopf algebras with a projection onto $H(2)$, 
by classifying braided Hopf algebras of dimension $3$ in ${}_\htw^\htw{\cal YD}$. The other case, namely of quasi-Hopf algebras $A$ with a projection that 
covers the embedding of a $3$-dimensional quasi-Hopf subalgebra of $A$, was considered in \cite{bm}.

It is well known that, up to isomorphism, the $6$-dimensional Hopf algebras are: $k[C_6]$, the group Hopf algebra of the cyclic group 
of order $6$; $k[S_3]$, the  group Hopf algebra of the symmetric group with 3 letters; the dual Hopf algebra of $k[S_3]$, denoted by $k^{S_3}$ (see e.g. \cite{ds}). 
As we will see, this classification yields for free the classification of the $3$-dimensional Hopf algebras  in ${}_{k[C_2]}^{k[C_2]}{\cal YD}$. 
Once more, we refer to \cite{db, bn} for the one to one correspondence between braided Hopf algebras in ${}_H^H{\cal YD}$ and quasi-Hopf algebra 
projections $
\xymatrix{
H \ar[r]<2pt>^i &\ar[l]<2pt>^{\pi} A
}
$
satisfying $\pi i=\Id_H$.  

\begin{proposition}\label{3dYD}
    Let $H=k[C_2]$. If $B$ is a $3$-dimensional braided Hopf algebra in ${}_{H}^{H}{\cal YD}$, then $B$ is isomorphic to one of the following braided Hopf algebras:
    
    $\bullet$ $B_{C_6}$, the group Hopf algebra $ k[C_3] $ viewed as an object of ${}_{H}^{H}{\cal YD}$ via trivial $H$-action and $H$-coaction.

    $\bullet$ $B_{S_3}$, the group Hopf algebra $ k[C_3] $ considered in ${}_{H}^{H}{\cal YD}$ via the trivial $H$-coaction and $H$-action defined by 
		$g\cd x=x^2$, where $x$ stands for a generator of $C_3$.
    
    $\bullet$ $B_{*}$, the group Hopf algebra $k[C_3]$ regarded as an object of ${}_{H}^{H}{\cal YD}$ via the trivial $H$-action and $H$-coaction 
		determined by 
    \begin{equation}
       \l(x^i)=p_+\ot x^i + p_-\ot x^{(2i)'},
    \end{equation}
    where $x$ is the generator of $C_3$ and $(2i)'$ is the remainder of the division of $2i$ by $3$, $i\in\{0, 1, 2\}$.

    Moreover, via the biproduct construction, to $B_{C_6}$, $B_{S_3}$ and $B_{*}$ correspond $k[C_6]$, $k[S_3]$ and $k^{S_3}$, respectively.
\end{proposition}
\begin{proof}
The group isomorphism $C_6\simeq C_3\times C_2$ produces a Hopf algebra isomorphism 
$K[C_6]\simeq k[C_3\times C_2]\simeq k[C_3]\ot k[C_2]$, where $k[C_3]\ot k[C_2]$ has the tensor product Hopf algebra structure. It follows form here 
that $k[C_6]$ can be realized as the biproduct Hopf algebra between $B_{C_6}$ and $k[C_2]$.

The symmetric group $S_3$ identifies to the group semidirect product $C_3\tie C_2$. Actually, if $\sigma\in S_3$ is a cycle of length $3$ and $\tau\in S_3$ is a 
transposition, $C_3=\le \sigma\ri$, $C_2=\le \tau\ri$ and $\tau\sigma=\sigma^{-1}\tau$. In Hopf algebra language, the inclusion $i: k[C_2]\ra k[S_3]$ 
is a Hopf algebra morphism and $\pi: k[S_3]\ra k[C_2]$ given by $\pi(\sigma^i\tau^j)=\tau^j$, $0\leq i\leq 2$ and $0\leq j\leq 1$, is a Hopf algebra 
morphism that covers $i$. Thus, $k[S_3]$ is a biproduct Hopf algebra between a braided Hopf algebra $B$ within ${}_{H}^{H}{\cal YD}$ and $k[C_2]$. 
As $B={\rm Im}(E)$, with $E: k[S_3]\ra k[S_3]$ given by $E(\sigma^i\tau^j)=\sigma^i\tau^jS(\pi(\sigma^i\tau^j))=\sigma^i$, for all 
$0\leq i\leq 2$ and $0\leq j\leq 1$, we obtain that $B=k[C_3]$ as a $k$-vector space, and an object of ${}_{H}^{H}{\cal YD}$ with structure determined by 
$\tau\tr \sigma^i=\tau\sigma^i\tau=\sigma^{-i}=\sigma^{(2i)'}$ and $\sigma^i\mapsto \pi(\sigma^i)\ot \sigma^i=1\ot \sigma^i$, for all $0\leq i\leq 2$. The structure of $B$ as a Hopf algebra in ${}_{H}^{H}{\cal YD}$ equals the structure of the group Hopf algebra $k[C_3]$, and therefore $B=B_{S_3}$.  
   
With notation as above, we have a Hopf algebra morphism $i^*: k^{S_3}\ra k[C_2]^*\equiv k[C_2]$ that covers the Hopf algebra inclusion 
$\pi^*: k[C_2]\equiv k[C_2]^*\ra k^{S_3}$; the Hopf algebra identification $k[C_2]^*\equiv k[C_2]$ is given by $1\mapsto P_1 + P_{\tau}$ and 
$\tau\mapsto P_1-P_{\tau}$, where $\{P_1, P_{\tau}\}$ is the basis of $k[C_2]^*$ dual to the basis $\{1, \tau\}$ of $k[C_2]$. Note that, the inverse of 
this correspondence is determined by $P_1\mapsto p_+$ and $P_\tau\mapsto p_-$. So, as in the previous case, $k^{S_3}$ can be characterized as a 
biproduct Hopf algebra between a braided Hopf algebra $B'$ in ${}_{H}^{H}{\cal YD}$ and $k[C_2]$. Explicitly, if $\{P_{ij}\}$ is the basis of 
$k^{S_3}$ dual to the basis $\{\sigma^i\tau^j\}$ of $k[S_3]$, we have that $\pi^*: k[C_2]\ra k^{S_3}$ is given by 
$\pi^*(1)=\va$ and $\pi^*(\tau)=\sum_{i, j=0}^{2, 1}(-1)^jP_{ij}$, where $\va$ is the counit of $k[S_3]$ or, alternatively, the unit of $k^{S_3}$. 
Likewise, $i^*: k^{S_3}\ra k[C_2]$ is determined by $i^*(P_{i0})=\delta_{i, 0}p_+$ and $i^*(P_{i1})=\delta_{i, 0}p_-$, $0\leq i\leq 2$, where 
$\delta_{i, j}$ is the Kronecker's delta of $i$ and $j$. By using the definition of the projection $E: k^{S_3}\ra B'$, we find that 
$E(P_{i0})=P_{i0}+P_{i1}$ and $E(P_{i1})=0$, for all $0\leq i\leq 2$, fact which implies that $B'$ is spanned by $\{P_{i0}+P_{i1}\mid 0\leq i\leq 2\}$. 
Moreover, if we denote $e_i:=P_{i0}+P_{i1}$, $0\leq i\leq 2$, then $\{e_0, e_1, e_2\}$ is a basis of $B'$ defined by idempotents such that 
$e_0+e_1+e_2=1$. In addition, if we consider $C_3=\le \sigma\ri$, $e_i(\sigma^{l})=\delta_{i, l}$, for all $0\leq i, l\leq 2$, then 
$\{e_0, e_1, e_2\}$ can be seen as the dual basis of the basis $\{1, \sigma, \sigma^2\}$ of $k[C_3]$. More explicitly, the isomorphism 
between $B'$ and $k[C_3]^*\cong k[C_3]$ can be uncovered as follows.  

Let $\omega$ be a primitive root of unity of order $3$ and set $x:=e_0+\omega e_1+\omega^2 e_2$. We compute that $x^2:=e_0+\omega^2 e_1+ \omega e_2$ 
and $x^3=e_0+e_1+e_2=1$. Since $\{1, x, x^2\}$ is as well a basis of $B'$, we conclude that $B'\simeq k[C_3]=k[\le x\ri]$ as $k$ algebras. 
Furthermore, the braided coalgebra structure of $B'$ in ${}_{H}^{H}{\cal YD}$ is determined by 
\[
\un{\Delta}(e_i)=\sum\limits_{(u+v)'=i} e_u\ot e_v,~\forall~i\in\{0, 1, 2\},
\]   
where the sum is over $u, v\in \{0, 1, 2\}$ such that the remainder $(u+v)'$ of the division of $u+v$ by $3$ equals $i$. A simple inspection guarantees us 
that $\un{\Delta}(x)=x\ot x$, and therefore $B'\simeq k[C_3]$ as Hopf algebras. Finally, the $H$-action on $B'$ is trivial becasue $B'$ is a commutative 
algebra and $\tau$ has order $2$. Concerning the left $H$-coaction on $B'$, it is given by 
$e_i\mapsto (i^*\ot \Id_{B'})(\Delta(e_i))=p_+\ot e_i + p_-\ot e_{(2i)'}$, for all $0\leq i\leq 2$, where $\Delta$ is the comultiplication of $k^{S_3}$. 
In terms of $x$, we have $x^i\mapsto p_+\ot x^i + p_-\ot x^{(2i)'}$, for all $0\leq i\leq 2$. 

Hence, any $6$-dimensional Hopf algebra is a biproduct Hopf algebra, and so determined by a braided Hopf algebra in ${}_H^H{\cal YD}$. 
The classification result mentioned above say then that $B_{C_6}$, $B_{S_3}$ and $B_*$ are the only types of $3$-dimensional braided Hopf algebras in 
${}_H^H{\cal YD}$. 
\end{proof}

It is well-known that, for a finite dimensional Hopf algebra $H$, the category of left-right Yetter-Drinfeld modules over $H$, ${}_H{\cal YD}^H$, is isomorphic 
(even as a braided category) to the category of left modules over the quantum double $D(H)$ of $H$; more details can be found in \cite{kas}. As far as we are 
concerned, if $M\in {}_H{\cal YD}^H$ then $M$ is a left $D(H)$-module via the action 
$(h^*\bowtie h)\cdot m=h^*((h\cdot m)_{(1)})(h\cdot m)_{(0)}$, for all $h^*\in H^*$ and $h\in H$, where $\cdot$ is the left $H$-action on $M$ and 
$M\ni m\mapsto m_{(0)}\ot m_{(1)}\in M\ot H$ is our sigma notation for the right action of $H$ on $M$. Conversely, any left $D(H)$-module $M$ 
can be seen in ${}_H{\cal YD}^H$ via the following structure ($h\in H$ and $m\in M$),
\[
h\cdot m=(\va\bowtie h)m~~\mbox{and}~~m\mapsto \sum\limits_{i=1}^n(e^i\bowtie 1)m\ot e_i, 
\]
where $\{e^i, e_i\}$ are dual bases in $H^*$ and $H$, and the juxtaposition is the action of $D(H)$ on $M$.

On the other hand, by the Hopf version of \cite[Theorem 8.14]{bcpvo}, the categories ${}_H^H{\cal YD}$ and ${}_H{\cal YD}^H$ are monoidally isomorphic.  
The inverse functors that provide the isomorphism act as identity on modules and morphisms, and transport the left (resp. right) $H$-coaction into a right (resp. left) 
one with the help of the inverse of the antipode $S$ of $H$ (resp. with the help of $S$). Consequently, ${}_H^H{\cal YD}$ in (monoidally) isomorphic to the category of left $D(H)$-modules.  
 
For $H=k[C_2]$, one can check easily that $D(\htw)=k[C_2]^*\ot k[C_2]\cong k[C_2]\otimes k[C_2]\cong k[C_2\times C_2]\cong k[C_4]$ as a Hopf algebra. Thus, in this particular case we have four non-isomorphic simple $D(H)$-modules, and therefore four non-isomorphic simple left Yetter-Drinfeld modules over $H$ which will be denoted in what follows by $M^i_j$, $i, j\in \{0, 1\}$. As the $M_i^j$'s are all one dimensional, we consider $M^i_j=k\le m_j^i\ri$ for some non-zero vector $m_i^j$.     
By using the correspondences mentioned above, we have the following left Yetter-Drinfeld module structure for $M_i^j$:  
\begin{equation}
    g\cd m_j^i=(-1)^jm_j^i\text{ and }\lambda(m_j^i)=g^i\ot m_j^i.
\end{equation}

With notation as in Proposition \ref{3dYD}, we have that  
\begin{eqnarray}
&& B_{S_3} = k 1 \oplus k(x + x^2) \oplus k(x - x^2)= M_0^0 \oplus M_0^0 \oplus M_1^0 \text{ and }\\
&& B_*=  k (1 + x + x^2) \oplus k(x + x^2) \oplus k(x - x^2)=M_0^0 \oplus M_0^0 \oplus M_0^1.
\end{eqnarray}

We move now to the quasi-Hopf setting. In what follows, $\mfq$ is a forth root of $1$ in $k$, so $\mfq^2=-1$.  

\begin{proposition}\label{simp mod htw}
    ${}_\htw^\htw{\cal YD}$ is a semisimple monoidal category, with $4$ simple objects, let's say $M_i$, $0\leq i\leq 3$. Each  
		$M_i$ is one dimensional and if $m_i$ is a generator of $M_i$, $M_i$ is in ${}_H^H{\cal YD}$ with structure given by 
    \begin{equation}
        g\cd m_i = (-1)^i m_i \quad\text{and}\quad \l(m_i)=(p_+ + \mfq^i p_-)\ot m_i.
    \end{equation}
\end{proposition} 
\begin{proof}

     By \cite[Proposition 2.10]{bct}, $D(\htw)$ is generated as an algebra by $X=\va\bowtie g$ and $Y=(P_1-P_g)\bowtie 1$; here $\{P_1, P_g\}$ is the basis 
		of $k[C_2]^*$ dual to the basis $\{1, g\}$ of $k[C_2]$. $X, Y$ verify $X^2=1$ and $Y^2=X$, so $D(H(2))=k[\le Y\ri]\cong k[C_4]$ as an algebra. Then, the 
	quasi-Hopf algebra structure of $D(\htw)$ is defined by 
		\begin{eqnarray*}
		&&\Delta(Y)=-\frac{1}{2}(Y\ot Y + Y^3\ot Y + Y\ot Y^3 - Y^2\ot Y^3),~~\va(Y)=-1,\\
		&&\Phi=1-2p_{-}^X\ot p^X_{-}\ot p^X_{-},~~S_D(Y)=Y,~~\a_D=X,~~\b_D=1,
		\end{eqnarray*}
		where $p_{-}^X=\frac{1}{2}(1-X)$; note that $X=Y^2$ is a grouplike element of $D(\htw)$ of order two, so $p^X_-$ is an idempotent element of $D(\htw)$. 
		
As in the Hopf case, we have an isomorphism of categories ${}_{\htw}^{\htw}{\cal YD}\simeq {}_{\htw}{\cal YD}^{\htw}\simeq {}_{D(\htw)}{\cal M}\simeq 
{}_{k[C_4]}{\cal M}$; see \cite[Theorems 8.14 and 8.29]{bcpvo}. Since ${}_{k[C_4]}{\cal M}$ is a semisimple category with four non-isomorphic simple objects, 
we deduce that ${}_{\htw}^{\htw}{\cal YD}$ is a semisimple category with four non-isomorphic simple objects, too.
		
Denoted by $M_j$, $j\in \{0,1,2,3\}$, the simple modules over $k[C_4]$. If $m\in M_j$, then $Y\cdot m= \mfq^j m$. By moving forward, we describe the four 
simple objects of ${}_{\htw}^{\htw}{\cal YD}$.  
     
    The correspondence in \cite[Lemma 8.28]{bcpvo} allows us to compute the simple objects of ${}_{\htw}{\cal YD}^\htw$, still denoted by $\{M_j\}_j$. 
		Since $\htw=k[\le X\ri]\cong k[C_2]$ as a subalgebra of $D(\htw)$, the left $\htw$-action on $M_j$ is given, for all $m\in M_j$, by
    \begin{equation*}
        X \cdot m = Y^2 \cdot m = \mfq^{2j} m = (-1)^j m.
    \end{equation*}
  Also, by the quoted result, the right $\htw$-coaction on $M_j$ is defined by 
    \[
        \rho(m)= (P_1 \bowtie p^1_2)m \ot S^{-1 }(p^2)p^1_1 + (P_1 \bowtie \tilde{p}^1_2)m \ot S^{-1 }(p^2)gp^1_1, 
     \]
	where $p_R=p^1\ot p^2=x^1\ot x^2\b S(x^3)=1-2p_{-}\ot p_{-}$. It is immediate that $\Delta(p^1)\ot p^2=g\ot g\ot p_- + 1\ot 1\ot p_+$, hence 
	\begin{eqnarray*}
			\rho(m)	&=&- (P_1\bowtie g)m\ot p_- + (P_g\bowtie g)m\ot p_- + (P_1\bowtie 1)m\ot p_+ + (P_g\bowtie 1)m\ot p_+\\
			        &=& -((P_1-P_g)\bowtie g)m\ot p_- + ((P_1+P_g)\bowtie 1)m\ot p_+\\
							&=& -(Y(\va\bowtie g))m\ot p_- + (\va\bowtie 1)m\ot p_+\\
							&=& -Y^3m\ot p_- + m\ot p_+\\
							&=&m\ot p_{+}-\mfq^{3j}m\ot p_{-}=m\ot p_{+} - \mfq^{-j}p_{-}.
    \end{eqnarray*}

    To achieve our aim, we use the functor described in \cite[Theorem 8.14]{bcpvo} in order to get the left 
		Yetter-Drinfeld structure over $\htw$ of $M_j$. To this end, we need the elements  
    \begin{eqnarray*}
        q_R&=&q^1 \ot q^2:=X^1\ot S^{-1}(\a X^3)X^2= 1\ot g + 2p_-\ot p_-=1\ot p_+ - g\ot p_-~~{\rm and}\\
        p_L&=&\tilde{p}^1\ot \tilde{p}^2:=X^2S^{-1}(X^1\b)\ot X^3=1\ot p_+ + g\ot p_-.
				\end{eqnarray*}
		A simple computation ensures us that 		
	\[			
	(\Delta(q^1)\ot q^2)\Phi^{-1}=1 - 2p_+\ot p_+ \ot p_-
  \]
	and since $p_\pm\cdot m=\frac{1}{2}(1\pm g^{2j})m$, for all $m\in M_j$, we get that, for all $m\in M_j$, 
	\begin{eqnarray*}
	&&\hspace{-1cm}
	(\widetilde{p}^2\cdot m)_{(1)}\tilde{p}^1\ot (\tilde{p}^2\cdot m)_{(0)}\\
	&=&(p_+\cdot m)_{(1)}\ot (p_+\cdot m)_{(0)} + (p_-\cdot m)_{(1)}g\ot (p_-\cdot m)_{(0)}\\
	&=&\frac{1}{2}(1+\mfq^{2j})(p_+-q^{-j}p_-)\ot m + \frac{1}{2}(1-\mfq^{2j})(p_+ - \mfq^{-j}p_-)g\ot m\\
	&=&(p_+ - \mfq^jp_-)\ot m.
	\end{eqnarray*}
	 Now, by \cite[Lemma 8.28]{bcpvo}, the left $\htw$-coaction on $M_j$ is given, for all $m\in M_j$, by  
    \begin{eqnarray*}
        \lambda(m)&=& q^1_1 x^1 S(q^2 x^3 (\tilde p^2 \cdot m)_1 \tilde p^1)\ot q_2^1 x^2 \cdot (\tilde p^2 \cdot m)_0\\
                  &=& S((\tilde p^2 \cdot m)_{(1)} \tilde p^1)\ot (\tilde p^2 \cdot m)_{(0)} - 2p_+
				S(p_-(\tilde p^2 \cdot m)_{(1)} \tilde p^1)\ot p_+\cdot (\tilde p^2 \cdot m)_{(0)}\\
				          &=&(p_+-\mfq^{-j}p_-)\ot m.
    \end{eqnarray*}
  The left $\htw$-module structure on $M_j$, considered as a left Yetter-Drinfeld module, remains the same: $g\cdot m=(-1)^jm$, for all $m\in M_j$. 
  By switching $1$ and $3$ and, respectively, $0$ and $3$, we obtain the structure of the simple objects $\{ M_j\}_{0\leq j\leq 3}$ of 
	${}_{\htw}^{\htw}{\cal YD}$ as in the statement of the proposition.
\end{proof}

\begin{corollary}\label{YDH2 3casesforB}
    Let $B$ be a $3$-dimensional bialgebra in ${}_{\htw}^{\htw}{\cal YD}$. $B$ is isomorphic as a Yetter-Drinfeld module to one of the following:
    \begin{eqnarray}
        && \label{3ydh2c1} M_0 \oplus M_{2i} \oplus M_{2j};\\
        && \label{3ydh2c2} M_0 \oplus M_{2i} \oplus M_{2j+1};\\
        && \label{3ydh2c3} M_0 \oplus M_{2i+1} \oplus M_{2j+1},  
    \end{eqnarray} 
    where $i,j\in\{ 0,1\}$ and $\{M_l\}_l$ are as in Proposition \ref{simp mod htw}.
\end{corollary}
\begin{proof}  
    The Corollary follows directly from Proposition \ref{simp mod htw}, by considering  the last two relations in (\ref{modalg1}), 
		which imply that one of the  Yetter-Drinfeld submodules of $B$, containing the unit, must be isomorphic to $M_0$.
\end{proof}

For a braided bialgebra, the case (\ref{3ydh2c3}) cannot occur. 

\begin{lemma}\label{lemmaYDH2 case 3}
    Let $B$ be a $3$-dimensional bialgebra in ${}_{\htw}^{\htw}{\cal YD}$. In the notation of Proposition \ref{simp mod htw}, $B$ is not isomorphic as a Yetter-Drinfeld module to $M_0 \oplus M_{2i+1} \oplus M_{2j+1}$.
\end{lemma}
\begin{proof}
    Suppose the contrary. Namely, that $B=k 1 \oplus k x \oplus k y$ is a bialgebra in ${}_{\htw}^{\htw}{\cal YD}$ with unit $1$ and $x, y\in B$ such that 
		$kx\simeq M_{2i+1}$ and $ky\simeq M_{2j+1}$ as left Yetter-Drinfeld $\htw$-modules, for some $i, j\in \{0,1\}$. 
    We characterize the braided bialgebra structure of $B$ in terms of its generators $x,y$. For instance, 
    by applying (\ref{modalg1}) to $x, y$ we deduce that 
    \begin{gather}
        g\cdot x^2 = x^2,\quad g\cdot y^2 = y^2,\quad g\cdot xy = xy,\quad g\cdot yx = yx,
    \end{gather}
    implying that $x^2, y^2, xy, yx\in k1$; otherwise stated,  
    \begin{gather}
        x^2=\o 1,\quad y^2=\o' 1,\quad 
        xy= l 1,\quad yx=r1,
    \end{gather}
    for some $\o,\o',l,r\in k$. Secondly, $p_+\cdot x=p_+\cdot y=0$ and $p_{-} \cdot x=x$, $p_{-}\cdot y=y$. These relations together with 
		(\ref{modalg2}) applied to the generators $x, y$ lead to the following equalities: 
    \begin{align}
        \o x=(xx)x=-x(xx)=-\o x&,\quad
        \o' y=(yy)y=-y(yy)= -\o' y;\\
        \o y = (xx)y=-x(xy)=-l x&,\quad
        rx=(yx)x=-y(xx)=-\o y;\\
        l x=(xy)x=-x(yx)=-rx&,\quad
        l y =(xy)y=-x(yy)=-\o' x;\\
        ry=(yx)y=-y(xy)=-ly &,\quad
        \o' x =(yy)x=-y(yx)=-ry.
    \end{align}
    Thus $l=r=\o=\o'=0$, and therefore $B$ is generated as an algebra by $x,y$ with relations $x^2=y^2=xy=yx=0$.
    
    From $\un{\va}_B(g\cdot b)=\va(g)\,\un{\va}_B(b)=\un{\va}_B(b)$, we get $\un{\va}_B(x)=\un{\va}_B(y)=0$.
    Thus, if we consider $\un{\Delta}_B(x)=\sum_{a,b\in\{1,x,y\}} \alp_{ab}\ a\ot b $, the counit properties say that 
    \begin{equation}
        \un{\Delta}_B(x)=1\ot x + x \ot 1 + \alp_{xx} x\ot x + \alp_{yy} y\ot y + \alp_{xy} x\ot y + \alp_{yx} y\ot x.
     \end{equation}
     
    The first relation in (\ref{YDcoal con1}) applied to $x$ entails to $-\un{\Delta}_B(x)=(g\ot g)\cdot \un{\Delta}_B(x)$, and is equivalent to
    \begin{equation}
        \begin{pmatrix}
        0 & -1 & 0\\
        -1 & -\alp_{xx} & -\alp_{xy}\\
        0 & -\alp_{yx} & -\alp_{yy}
    \end{pmatrix}
    = 
    \begin{pmatrix}
        0 & -1 & 0\\
        -1 & \alp_{xx} & \alp_{xy}\\
        0 & \alp_{yx} & \alp_{yy}
    \end{pmatrix}.
    \end{equation}
    We conclude that $\un{\Delta}(x)=x\ot 1 + 1\ot x$ and $\un{\va}(x)=0$, that is, $x$ is a primitive element. Consequently, by considering 
		(\ref{moltcon2}) for $x, x$, since $x^2=0$, $p_{-}\cdot x=x$ and $p_+\cdot x=0$, we get that 
    \begin{align*}
        0=\un{\Delta}_B(x^2)=
                &(y^1X^1\cdot 1)(y^2Y^1(x^1X^2\cdot x)_{[-1]}x^2X_1^3\cdot 1) \ot (y_1^3Y^2\cdot (x^1X^2\cdot x)_{[0]})(y_2^3Y^3x^3X_2^3\cdot x)\\
                &+ (y^1X^1\cdot x)(y^2Y^1(x^1X^2\cdot 1)_{[-1]}x^2X_1^3\cdot 1) \ot (y_1^3Y^2\cdot (x^1X^2\cdot 1)_{[0]})(y_2^3Y^3x^3X_2^3\cdot x)\\
                &+(y^1X^1\cdot 1)(y^2Y^1(x^1X^2\cdot x)_{[-1]}x^2X_1^3\cdot x) \ot (y_1^3Y^2\cdot (x^1X^2\cdot x)_{[0]})(y_2^3Y^3x^3X_2^3\cdot 1)\\
                &+ (y^1X^1\cdot x)(y^2Y^1(x^1X^2\cdot 1)_{[-1]}x^2X_1^3\cdot x) \ot (y_1^3Y^2\cdot (x^1X^2\cdot 1)_{[0]})(y_2^3Y^3x^3X_2^3\cdot 1)\\
                =
                &1\ot x^2 + x \ot x + x_{[-1]}  \cdot x \ot     x_{[0]} + x^2\ot 1 
                =    (1+ \mfq^{2i+1})x \ot     x.
    \end{align*}
	Hence, $1+\mfq^{2i+1}=0$ for a certain $i\in \{0, 1\}$, a contradiction. 	
\end{proof}

We study now the case (\ref{3ydh2c2}). For simplicity, we look first at the algebra case.  

\begin{lemma}\label{lemmaYDH2 algebra case 2}
    Let $B$ be a $3$-dimensional algebra in ${}_{\htw}^{\htw}{\cal YD}$, isomorphic as a Yetter-Drinfeld module to $M_0 \oplus M_{2i} \oplus M_{2j+1}$. 
		Then $B$ is isomorphic to one of the following algebras in ${}_{\htw}^{\htw}{\cal YD}$: 
		
		$\bullet$ $B_{2o}^{00}:=k1\oplus k\bar x\oplus k\bar y\simeq M_0\oplus M_2\oplus M_{2j+1}$ in ${}_{\htw}^{\htw}{\cal YD}$, with relations 
		$\bar x^2=\bar y^2=\bar x\bar y=\bar y\bar x=0$;
		
		$\bullet$ $B_{0o}^{10}:=k1\oplus k\bar x\oplus k\bar y\simeq M_0\oplus M_0\oplus M_{2j+1}$ in ${}_{\htw}^{\htw}{\cal YD}$, with relations $\bar x^2=1$, 
		$\bar y^2=0$, $\bar x\bar y=\bar y$, $\bar y\bar x=-\bar y$;
		
		$\bullet$ $B_{0o}^{0}:=k1\oplus k\bar x\oplus k\bar y\simeq M_0\oplus M_0\oplus M_{2j+1}$ in ${}_{\htw}^{\htw}{\cal YD}$, with relations $\bar x^2=0$, 
		$\bar y^2=\bar x$, $\bar x\bar y=\bar y\bar x=0$;
		
		$\bullet$ $B_{0o}^{00}=k1\oplus k\bar x\oplus k\bar y\simeq M_0\oplus M_0\oplus M_{2j+1}$ in ${}_{\htw}^{\htw}{\cal YD}$, with relations $\bar x^2=0$, 
		$\bar y^2=0$, $\bar x\bar y=\bar y\bar x=\bar y$;

        $\bullet$ $B_{0o}^{10c}:=k1\oplus k\bar x\oplus k\bar y\simeq M_0\oplus M_0\oplus M_{2j+1}$ in ${}_{\htw}^{\htw}{\cal YD}$,
        with relations $\bar x^2=1$, 
		$\bar y^2=0$, $\bar x\bar y=\bar y\bar x=\bar y$. 
\end{lemma}
\begin{proof}
    We use arguments similar to the ones in the proof of Lemma \ref{lemmaYDH2 case 3}. Assume that 
		$B=k1\oplus kx\oplus ky\simeq M_0\oplus M_{2i}\oplus M_{2j+1}$ in ${}_{\htw}^{\htw}{\cal YD}$, for some $i,j\in \{0,1\}$. 
		
		We have $g\cdot x=x$ and $g\cdot y=-y$, so the first relation in (\ref{modalg1}) yields $g\cdot x^2=x^2$ and $g\cdot y^2=y^2$. 
		Thus, $x^2=\alpha 1 + \beta x$ and $y^2=\ov{\alpha}1+\ov{\beta}x$, for some scalars $\alpha, \ov{\alpha}, \beta, \ov{\beta}\in k$. 
    Likewise, $g\cdot (xy)=-xy$ and $g\cdot (yx)=-yx$, hence $xy=\gamma y$ and $yx=\ov{\gamma}y$, for some $\gamma, \ov{\gamma}\in k$. 
		
Clearly, $p_-\cdot x=0$ and $p_-\cdot y=y$. Therefore, the associativity of the multiplication of $B$, expressed by the equations in (\ref{modalg2}), is equivalent to the following equalities:
\begin{eqnarray*}
&&(xx)x=x(xx),~(yy)y=-y(yy),~(yx)x=y(xx),~~(xx)y=x(xy);\\
&&(xy)x=x(yx),~(xy)y=x(yy),~(yx)y=y(xy), (yy)x=y(yx).
\end{eqnarray*}
These are satisfied if and only if 
\begin{equation}\eqlabel{concomalghtw}
\begin{array}{l}
\ov{\beta}(\gamma + \ov{\gamma})=-2\ov{\alpha},~\ov{\gamma}^2=\alpha +\beta\ov{\gamma},~\gamma^2=\alpha +\beta\gamma;\\
\alpha\ov{\beta}=\ov{\alpha}\gamma=\ov{\alpha}\ov{\gamma},~\ov{\alpha}+\beta\ov{\beta}=\ov{\beta}\gamma=\ov{\beta}\ov{\gamma}.
\end{array}
\end{equation}

It remains to study the colinearity of the multiplication of $B$. By (\ref{modalg3}), applied to our structure, we find that 
\begin{eqnarray*}
\l_B(x^2)&=&X^1(x^1Y^1\cdot x)_{[-1]}x^2(Y^2\cdot x)_{[-1]}Y^3\ot (X^2\cdot (x^1Y^1\cdot x)_{[0]})(X^3x^3\cdot (Y^2\cdot x)_{[0]})\\
&=&X^1(p_+ + (-1)^ip_-)^2\ot (X^2\cdot x)(X^3\cdot x)=1\ot x^2,
\end{eqnarray*}
and so $\beta (p_+ + (-1)^ip_-)\ot x=\beta 1\ot x$; we conclude that $i=0$ or $\beta=0$. 

In a similar way, we compute that 
\begin{eqnarray*}
\l_B(y^2)&=&X^1(x^1\cdot y)_{[-1]}x^2y_{[-1]}g\ot (X^2\cdot (x^1\cdot y)_{[0]})(X^3x^3\cdot y_{[0]})\\
&=&
(p_+ - p_-)\ot y^2 + 2 p_- \ot y^2  =1\ot y^2,
\end{eqnarray*}    
hence $\ov{\beta}(p_+ + (-1)^ip_-)\ot x=\ov{\beta}1\ot x$; so, either $i=0$ or $\ov{\beta}=0$. 
 
By using (\ref{modalg3}) for $x$ and $y$, we get $\gamma(p_+ + \mfq^{2j+1}p_-)\ot y=\l_B(xy)=\gamma(p_+ +\mfq^{2(i+j)+1}p_-)\ot y$, 
and therefore $i=0$ or $\gamma=0$. Similarly, $\ov{\gamma}(p_+ + \mfq^{2j+1}p_-)\ot y=\l_B(yx)=\ov{\gamma}(p_+ + \mfq^{2(j+i)+1}p_-)\ot y$, 
which entails to $i=0$ or $\ov{\gamma}=0$.

\un{Case 1}: i=1. We get $\beta=\ov{\beta}=\gamma=\ov{\gamma}=0$, implying $\alpha=\ov{\alpha}=0$, too. Hence $B=B_{2o}^{00}$.

\un{Case 2}: i=0. We get the relations stated in \equref{concomalghtw}; we distinguished the following possibilities: 

\un{2.1}: $\gamma\not=\ov{\gamma}$. Then $\ov{\beta}=0$, so $\ov{\alpha}=0$, too. In addition, $\ov{\gamma}=\beta - \gamma$, 
$\gamma^2=\alpha + \beta\gamma$, $4\alpha\not =-\beta^2$ and $x, y$ verify $x^2=\alpha 1 +\beta x$, $y^2=0$, $xy=\gamma y$ and 
$yx=(\beta-\gamma)y$. If we set $x':=\frac{\beta}{2}1 - x$ and $\delta:=\frac{\beta}{2} - \gamma$ then $\delta\not=0$ and 
$x'^2=\delta^21$, $x'y=\delta y$ and $yx'=-\delta y$. Take now $\ov{x}:=\delta^{-1}x'$ to get $\ov{x}^2=1$, $y^2=0$, $\ov{x}y=y$ and 
$y\ov{x}=-y$. It is clear at this point that the correspondence defined by $\ov{x}\mapsto \delta^{-1}x'=\delta^{-1}(\frac{\beta}{2}1 -x)$ and 
$\ov{y}\mapsto y$ defines an algebra isomorphism between $B_{0o}^{10}$ and $B$ in ${}_{\htw}^{\htw}{\cal YD}$. Note that 
$B_{0o}^{10}\simeq B=k1\oplus k(\delta^{-1}(\frac{\beta}{2}1-x))\oplus ky\simeq M_0\oplus M_0\oplus M_{2j+1}$ in ${}_{\htw}^{\htw}{\cal YD}$.   


\un{2.2}: $\gamma=\ov{\gamma}$. We are reducing in this case to $\ov{\beta}\gamma=-\ov{\alpha}$, $\gamma^2=\alpha +\beta\gamma$, $\alpha\ov{\beta}=\ov{\alpha}\gamma$ 
and $\ov{\alpha} + \beta\ov{\beta}=\ov{\beta}\gamma$. 

Assume $\ov{\beta}\not=0$. We have $\gamma=\ov{\gamma}=-\frac{\ov{\alpha}}{\ov{\beta}}$, $\alpha=-\big(\frac{\ov{\alpha}}{\ov{\beta}}\big)^2$, 
$\beta=-\frac{2\ov{\alpha}}{\ov{\beta}}$ and $\gamma^2=\alpha +\beta\gamma$ is automatically satisfied. Then $B$ is generated by $x, y$ with relations 
$\big(x+\frac{\ov{\alpha}}{\ov{\beta}}1\big)^2=0$, $y^2=\ov{\alpha}1+\ov{\beta}x$ and $xy=yx=-\frac{\ov{\alpha}}{\ov{\beta}}$. Hence, 
if we take $x':=x+\frac{\ov{\alpha}}{\ov{\beta}}1$ then $x'^2=0$, $y^2=\ov{\beta}x'$ and $x'y=yx'=0$, with $\ov{\beta}\not=0$. 
We deduce from here that the correspondence $\bar x\mapsto \ov{\beta}x'=\ov{\alpha}1+\ov{\beta}x=y^2$, $\bar y\mapsto y$, 
produces an isomorphism of algebras between $B_{0o}^{0}$ and $B$ in ${}_{\htw}^{\htw}{\cal YD}$. Note that 
$B_{0o}^0\simeq B=k1\oplus ky^2\oplus ky\simeq M_0\oplus M_0\oplus M_{2j+1}$ in ${}_{\htw}^{\htw}{\cal YD}$.
%

In the case when $\ov{\beta}=0$, we get $\ov{\alpha}=0$, so $B$ is generated by $x, y$ with relations $x^2=\alpha 1+\beta x$, $y^2=0$ 
and $xy=yx=\gamma y$, where $\alpha, \beta, \gamma$ are scalars obeying $\gamma^2=\alpha + \beta \gamma$. 
For $x':=\frac{\beta}{2}1-x$ and $\delta:= \frac{\beta}{2} - \gamma$, we have that 
$x'^2=\delta^2 1$ and $x'y=yx'=\delta 1$. If we have $\beta^2=-4\alpha$, this leads to $\delta^2=0$, and thus the correspondence 
$\bar x\mapsto x'=\frac{\beta}{2}1-x$ and $\bar y\mapsto y$ defines an 
algebra isomorphism between $B^{00}_{0o}$ and $B$ in ${}_{\htw}^{\htw}{\cal YD}$. Note that, in this case 
$B_{0o}^{00}\simeq B=k1\oplus k(\frac{\beta}{2}1-x)\oplus ky\simeq M_0\oplus M_0\oplus M_{2j+1}$ in ${}_{\htw}^{\htw}{\cal YD}$. 
Finally, if we have that $\beta^2\neq-4\alpha$, and therefore $\delta\neq 0$, we take $\ov x := \delta^{-1}x'$ which satisfies   
${\ov x}^2=1$, $\ov x y = y\ov x = y$ and $y^2=0$. We conclude that the correspondence $\bar x\mapsto \delta^{-1} x'=\delta^{-1} (\frac{\beta}{2}1-x)$ and $\bar y\mapsto y$ defines an   algebra isomorphism between $B^{10c}_{0o}$ and $B$ in ${}_{\htw}^{\htw}{\cal YD}$. Note that  $B_{0o}^{10c}\simeq B=k1\oplus k (\frac{\beta}{2}1-x)\oplus ky\simeq M_0\oplus M_0\oplus M_{2j+1}$ in ${}_{\htw}^{\htw}{\cal YD}$. 
\end{proof}

\begin{proposition}\label{lemmaYDH2 case 2}
    Let $B$ be a $3$-dimensional bialgebra in ${}_{\htw}^{\htw}{\cal YD}$.  If $B$ is isomorphic as a Yetter-Drinfeld module to $M_0 \oplus M_{2i} \oplus M_{2j+1}$ 
		then $B$ is isomorphic to one of the following braided bialgebras:
		
		(i) $B_{0o}^{10}$, with comultiplication and counit given by 
		\begin{equation}\eqlabel{1braidbialgB0o10}
		\begin{array}{l}
		\un{\Delta}(x)=-\frac{1}{2}(1\ot 1 - 1\ot x - x\ot 1 - x\ot x),~\un{\va}(x)=1;\\
		\un{\Delta}(y)=\frac{1}{2}(1\ot y + y\ot 1 + x\ot y + y\ot x),~\un{\va}(y)=0.
		\end{array}
		\end{equation}
		
		(ii) $B_{0o}^{10}$, with comultiplication and counit given by
		\begin{equation}\eqlabel{2braidbialgB0o10}
		\begin{array}{l}
		\un{\Delta}(x)=\frac{1}{2}(1\ot 1 + 1\ot x + x\ot 1 - x\ot x),~\un{\va}(x)=-1;\\
		\un{\Delta}(y)=\frac{1}{2}(1\ot y + y\ot 1 - x\ot y - y\ot x),~\un{\va}(y)=0.
		\end{array}
		\end{equation}

        (iii) $B_{0o}^{10c}$, with comultiplication and counit given by
		\begin{equation}\eqlabel{3braidbialgB0o10c}
		\begin{array}{l}
		\un{\Delta}(x)=\frac{1}{2}(1\ot 1 + 1\ot x + x\ot 1 - x\ot x),~\un{\va}(x)=-1;\\
		\un{\Delta}(y)=\frac{1}{2}(1\ot y + y\ot 1 - x\ot y - y\ot x),~\un{\va}(y)=0.
		\end{array}
		\end{equation}
 \end{proposition}
\begin{proof}   
		We use the results obtained in Lemma \ref{lemmaYDH2 algebra case 2}; for simplicity, we consider the generators of the braided bialgebra $B$ as being $x, y$, rather that $\bar x, \bar y$. Also, we set $\un{\Delta}_B(x)=\sum_{p,q\in\{1,x,y\}}\alp_{p,q}~p\ot q$ and $\un{\Delta}_B(y)=\sum_{p,q\in\{1,x,y\}}\tau_{p,q}~p\ot q$.
		
Since $g\cdot y=-1$, by the second relation in (\ref{YDcoal con1}) we obtain $\un{\va}_B(y)=0$. The equality $-\un{\Delta}(y)=\un{\Delta}(g\cdot y)=
(g\ot g)\un{\Delta}(y)$ implies $\tau_{1, 1}=\tau_{x, x}=\tau_{y, y}=0$ and $(1+(-1)^i)\tau_{1, x}=(1+(-1)^i)\tau_{x, 1}=0$, 
$(1+(-1)^{i+1})\tau_{x, y}=(1+(-1)^{i+1})\tau_{y, x}=0$. We use now the counit properties to get $\tau_{1, x}=\tau_{x, 1}=0$ 
and $\tau_{y, 1} +\un{\va}(x)\tau_{y, x}=1=\tau_{1, y}+ \un{\va}(x)\tau_{x, y}$. Summing up, we have 
\[
\un{\Delta}(y)=\tau_{1, y}1\ot y + \tau_{x, y}x\ot y + \tau_{y, 1} y\ot 1 + \tau_{y, x}y\ot x, 
\] 
with $(1+(-1)^{i+1})\tau_{x, y}=(1+(-1)^{i+1})\tau_{y, x}=0$ and $\tau_{y, 1} +\un{\va}(x)\tau_{y, x}=1=\tau_{1, y}+ \un{\va}(x)\tau_{x, y}$.

Denote $\un{\Delta}(y)=y_{\un{1}}\ot y_{\un{2}}=y'_{\un{1}}\ot y'_{\un{2}}$. With the help of (\ref{moltcon2}) we compute that 
\begin{eqnarray*}
\un{\Delta}(y^2)&=&(y^1X^1\cdot y_{\un{1}})(y^2Y^1(x^1X^2\cdot y_{\un{2}})_{[-1]}x^2X^3_1\cdot y'_{\un{1}})\ot 
(y^3_1Y^2\cdot (x^1X^2\cdot y_{\un{2}})_{[0]})(y^3_2Y^3x^3X^3_2\cdot y'_{\un{2}})\\
&=&\tau_{1, y}Y^1(x^1\cdot y)_{[-1]}x^2\cdot y'_{\un{1}}\ot (Y^2\cdot (x^1\cdot y)_{[0]})(Y^3x^3\cdot y'_{\un{2}})\\
&& 
+\tau_{x, y}x(Y^1(x^1\cdot y)_{[-1]}x^2\cdot y'_{\un{1}})\ot (Y^2\cdot (x^1\cdot y)_{[0]})(Y^3x^3\cdot y'_{\un{2}})\\
&&
+\tau_{y, 1}(y^1\cdot y)(y^2\cdot y'_{\un{1}})\ot y^3\cdot y'_{\un{2}} 
+\tau_{y, x}(y^1\cdot y)(y^2Y^1x_{[-1]}\cdot y'_{\un{1}})\ot (y^3_1Y^2\cdot x_{[0]})(y^3_2Y^3\cdot y'_{\un{2}})\\
&=&
\tau_{1, y}^2 1\ot y^2 + \tau_{1, y}\tau_{x, y}y_{[-1]}\cdot x\ot y_{[0]}y + \tau_{1, y}\tau_{y, 1}y_{[-1]}\cdot y\ot y_{[0]}
+ \tau_{1, y}\tau_{y, x}y_{[-1]}\cdot y\ot y_{[0]}x\\
&&+
\tau_{x, y}\tau_{1, y}x\ot y^2 + \tau_{x, y}^2x(y_{[-1]}\cdot x)\ot y_{[0]}y
+ \tau_{x, y}\tau_{y, 1}x(y_{[-1]}\cdot y)\ot y_{[0]}\\
&& 
+ \tau_{x, y}\tau_{y, x} x(y_{[-1]}\cdot y)\ot y_{[0]}x 
+ \tau_{y, 1}\tau_{1, y}y\ot y +\tau_{y, 1}\tau_{x, y}yx\ot y + \tau_{y, 1}^2y^2\ot 1\\
&& 
+ \tau_{y, 1}\tau_{y, x}y^2\ot x
+ \tau_{y, x}\tau_{1, y}y\ot xy + \tau_{y, x}\tau_{x, y} y(x_{[-1]}\cdot x)\ot x_{[0]}y\\
&& 
+ \tau_{y, x}\tau_{y, 1}(y^1\cdot y)(y^2x_{[-1]}\cdot y)\ot y^3\cdot x_{[0]} 
+ \tau^2_{y, x}(y^1\cdot y)(y^2x_{[-1]}\cdot y)\ot y^3\cdot (x_{[0]}x)\\
&=&
\tau_{1, y}^2 1\ot y^2 +\tau_{1, y}\tau_{x, y}\big(1 + \frac{1+(-1)^i}{2} + \mfq^{2j+1}\frac{1-(-1)^i}{2}\big)x\ot y^2 
+(1+\mfq^{2j+1})\tau_{1, y}\tau_{y, 1}y\ot y\\
&&
+ 2\tau_{y, 1}\tau_{y, x}y^2\ot x + \tau_{x, y}^2\big(\frac{1+(-1)^i}{2} +\mfq^{2j+1}\frac{1-(-1)^i}{2}\big)x^2\ot y^2
+ (-1)^i\tau_{y, x}^2y^2\ot x^2\\
&&
+\tau_{y, 1}^2y^2\ot 1 + \tau_{y, 1}\tau_{x, y}(\gamma\mfq^{2j+1} + \ov{\gamma})y\ot y
+ \tau_{1, y}\tau_{y, x}(\ov{\gamma}\mfq^{2j+1} + \gamma)y\ot y\\
&&
+\tau_{x, y}\tau_{y, x}\gamma\ov{\gamma}(1+\mfq^{2j+1})y\ot y.
\end{eqnarray*}

If $i=1$, $\tau_{x, y}=\tau_{y, x}=0$ and $\tau_{1, y}=\tau_{y, 1}=1$. So 
$\un{\Delta}(y^2)=1\ot y^2 + (1+\mfq^{2j+1})y\ot y + y^2\ot 1$, and this forces $y^2\not=0$, a contradiction. 
We conclude that $i=0$, and therefore 
\begin{eqnarray*}
&&\hspace{-5mm}
\un{\Delta}(y^2)=
\tau_{1, y}^2 1\ot y^2 +2\tau_{1, y}\tau_{x, y}x\ot y^2 
+(1+\mfq^{2j+1})\tau_{1, y}\tau_{y, 1}y\ot y
+ 2\tau_{y, 1}\tau_{y, x}y^2\ot x + \tau_{x, y}^2x^2\ot y^2\\
&&
\hspace{5mm}
+ \tau_{y, x}^2y^2\ot x^2
+\tau_{y, 1}^2y^2\ot 1 + \tau_{y, 1}\tau_{x, y}(\gamma\mfq^{2j+1} + \ov{\gamma})y\ot y 
+ \tau_{1, y}\tau_{y, x}(\ov{\gamma}\mfq^{2j+1} + \gamma)y\ot y\\
&&
\hspace{5mm}
+\tau_{x, y}\tau_{y, x}\gamma\ov{\gamma}(1+ \mfq^{2j+1})y\ot y.
\end{eqnarray*}

Similarly, by using the fact that $g\cdot x=x$ and $g\cdot y=-y$, we see that 
\[
\un{\Delta}(x)=\alp_{1, 1}1\ot 1 + \alp_{1, x}1\ot x + \alp_{x, 1}x\ot 1 + \alp_{x, x}x\ot x + \alp_{y, y}y\ot y
\]
with $\alp_{1, 1}=-\un{\va}(x)\alp_{1, x}$ and $\alp_{1, x}=\alp_{x, 1}=1-\un{\va}(x)\alp_{x, x}$. 

Owing to the above relations, the coassociativity of $\un{\Delta}$ on $x$ (see \equref{ydc2}) reduces to 
\[
\sigma_{y, y}\sigma_{x, 1}=\sigma_{y,y}\tau_{y, 1}=\sigma_{y, y}\tau_{1, y},~~\sigma_{x, x}\sigma_{y, y}=\sigma_{y, y}\tau_{x, y}=\sigma_{y, y}\tau_{y, x}.
\]
Likewise, the coassociativity of $\un{\Delta}$ on $y$ is equivalent to 
\[
\tau_{y, 1}\tau_{y, x}=\tau_{y, x}\sigma_{x, 1},~\sigma_{x, 1}\tau_{x, y}=\tau_{1, y}\tau_{x, y},~\tau_{x, y}^2=\sigma_{x, x}\tau_{x, y},~
\tau_{y, x}^2=\tau_{y, x}\sigma_{x, x},~\sigma_{y, y}\tau_{y, x}=-\sigma_{y, y}\tau_{x, y}.
\]

\un{Case 1}: $\sigma_{y, y}\not=0$. Then $\tau_{x, y}=\tau_{y, x}=\sigma_{x, x}=0$, $\tau_{1,y}=\tau_{y, 1}=\sigma_{x, 1}=\sigma_{1, x}=1$ and $\sigma_{1, 1}=-\un{\va}(x)$. 
Thus, $\Delta(y)=1\ot y + y\ot 1$ and $\Delta(y^2)=1\ot y^2 + y^2\ot 1 + (1+\mfq^{2j+1})y\ot y$. We cannot have $y^2=0$, so 
$y^2=x$. Consequently, $\un{\Delta}(x)=1\ot x + x\ot 1 + (1+\mfq^{2j+1})y\ot y$; together with $\un{\Delta}(x)=\sigma_{1, 1}1\ot 1 + 1\ot x + x\ot 1 
+ \sigma_{y, y}y\ot y$, this implies $\sigma_{1, 1}=\un{\va}(x)=0$ and $\sigma_{y, y}=1+\mfq^{2j+1}\not=0$. 

Summarizing, $B=B^{0}_{0o}$ as an algebra in ${}_{\htw}^{\htw}{\cal YD}$, and is a coalgebra in ${}_{\htw}{\cal M}$ 
with structure determined by 
\begin{eqnarray*}
&&\un{\Delta}(x)=x\ot 1 + 1\ot x + (1+\mfq^{2j+1})y\ot y,~ \un{\va}(x)=0,\\
&&\un{\Delta}(y)=1\ot y + y\ot 1,~\un{\va}(y)=0.
\end{eqnarray*}
Moreover, $\un{\Delta}(y^2)=\un{\Delta}(y)\un{\Delta}(y)$. Similar to the computation of $\un{\Delta}(y^2)$ in the above, one can find that 
\begin{eqnarray*}
&&\hspace{-7mm}\un{\Delta}(x)\un{\Delta}(x)\\
&=&xx_{\un{1}}\ot x_{\un{2}} + 
x_{\un{1}}\ot xx_{\un{2}} +\sigma_{y, y}(y^1\cdot y)(y^2Y^1(x^1\cdot y)_{[-1]}x^2\cdot x_{\un{1}})\ot (y^3_1Y^2\cdot (x^1\cdot y)_{[0]})
(y^3_2Y^3x^3\cdot x_{\un{2}})\\
&=&2x\ot x +\sigma_{y, y}y(y_{[-1]}\cdot x)\ot y_{[0]} + \sigma_{y, y}y\ot yx+ \\
&&
\sigma_{y, y}^2(y^1\cdot y)(y^2Y^1(x^1\cdot y)_{[-1]}x^2\cdot y)
\ot (y^3_1Y^2\cdot (x^1\cdot y)_{[0]})(y^3_2Y^3x^3\cdot y)\\
&=&2x\ot x - 2\mfq^{2j+1}(y^1\cdot y)(y^2Y^1y_{[-1]}\cdot y)\ot (y^3_1Y^2\cdot y_{[0]})(y^3_2Y^3\cdot y)\\
&=&2x\ot x + 2(y^1\cdot y)(y^2Y^1\cdot y)\ot y^3\cdot ((Y^2\cdot y)(Y^3\cdot y))\\
&=&2x\ot x - 2(y^1\cdot y)(y^2\cdot y)\ot y^3\cdot x=2x\ot x - 2y^2\ot x=0, 
\end{eqnarray*}
and therefore $\un{\Delta}(x^2)=\un{\Delta}(x)\un{\Delta}(x)$. Likewise, $\un{\Delta}(x)\un{\Delta}(y)=\mfq^{2j+1}(x\ot y -y\ot x)$ and 
$\un{\Delta}(y)\un{\Delta}(x)=\mfq^{2j+1}(y\ot x -x\ot y)$, so $\un{\Delta}$ is not multiplicative.

\un{Case 2}: $\sigma_{y, y}=0$. We start by noting that \equref{ydc3} is satisfied by $x, y$, so we will pay attention only to the other 
remaining relations. We distinguish two possibilities.

\un{21}: $\tau_{x, y}=0$. We reduce to $\tau_{1, y}=1$, $\tau_{yx}^2=\tau_{yx}\sigma_{xx}$ and $\tau_{y, 1}\tau_{y, x}=\tau_{y, x}\sigma_{x, 1}$. Furthermore, 
we cannot have $\tau_{y, x}=0$. Indeed, for $\tau_{y, x}=0$ we have $\tau_{y, 1}=1$ and $\un{\Delta}(y)=1\ot y+ y\ot 1$; thus $\un{\Delta}(y)\un{\Delta}(y)=
1\ot y^2 + y^2\ot 1 + (1+\mfq^{2j+1})y\ot y$, which implies $y^2\not=0$, and so $y^2=x$. We obtain that $\un{\Delta}(x)=\Delta(y^2)=
1\ot x + x\ot 1 + (1+\mfq^{2j+1})y\ot y$, hence $\sigma_{1, 1}=0$, $\sigma_{1, x}=\sigma_{x, 1}=1$ and $\sigma_{x, x}=1+\mfq^{2j+1}\not=0$. We get 
$\un{\va}(x)=0$, and this contradicts the fact that $x^2=1$. 

Thus, we have in this case that $\tau_{y, 1}=\sigma_{x, 1}=\sigma_{1, x}:=a$, $\tau_{y, x}=\sigma_{x, x}:=b$ and $\sigma_{1, 1}=-\un{\va}(x)a$, with 
$a+\un{\va}(x)b=1$, $b\not=0$. If $y^2=x$, $\gamma=\ov{\gamma}=0$ and $x^2=0$, from the general formula of $\un{\Delta}(y^2)$ computed above we deduce that 
$\un{\Delta}(x)=\un{\Delta}(y^2)=1\ot x + a(1+\mfq^{2j+1})y\ot y +2ab x\ot x + a^1x\ot 1$. Comparing with 
$\un{\Delta}(x)=\sigma_{1, 1}1\ot 1 + a1\ot x+ ax\ot 1 +bx\ot x$, we arrive at $\sigma_{1, 1}=0$, $a=0$, $a=1$ and $2ab=b$, which is impossible to achieve. 

We conclude that $y^2=0$, and by using again the formula for $\un{\Delta}(y^2)$ found above, 
we land at $a(1+\mfq^{2j+1}) + b(\ov{\gamma}\mfq^{2j+1} +\gamma)=0$. 

Suppose we are dealing with $B_{0o}^{00}$. Then $a=-b$ and $\sigma_{1, 1}=\un{\va}(x)b$. We compute that 
\begin{eqnarray*}
\un{\Delta}(x)\un{\Delta}(x)&=&\sigma_{1, 1}x_{\un{1}}\ot x_{\un{2}} - bx_{\un{1}}\ot x_{\un{2}} - bx_{\un{1}}\ot xx_{\un{2}} + bxx_{\un{1}}\ot xx_{\un{2}}\\
&=&\sigma_{1, 1}^21\ot 1 - 2b\sigma_{1, 1}1\ot x - 2b\sigma_{1, 1}x\ot 1 + (2b^2 + b\sigma_{1, 1})x\ot x,  
\end{eqnarray*}
which forces $\sigma_{1, 1}=0$ and $b=0$, a contradiction.

The only possibility left is $B=B_{0o}^{10}$, when we have $a=\mfq^{2j+1}b$ and 
\[
\un{\Delta}(x)\un{\Delta}(x)=(\sigma_{1, 1}^2-b^2)1\ot 1 + 2\mfq^{2j+1}(b+\sigma_{1, 1})(1\ot x + x\ot 1) - 2b(b - \sigma_{1, 1})x\ot x.
\]
As $x^2=1$ and $b\not=0$, $\sigma_{1, 1}^2=b^2 + 1$ and $\sigma_{1, 1}=b=0$; a contradiction. So this case is not possible.

\un{22}: $\tau_{x, y}\not=0$. Then $\sigma_{x, x}=\tau_{x, y}\not=0$, $\tau_{1, y}=\sigma_{x, 1}=\sigma_{1, x}$, $\tau_{y, x}^2=\tau_{y, x}\sigma_{x, x}$ and 
$\tau_{y, 1}\tau_{y, x}=\tau_{y, x}\sigma_{x, 1}$.   

$\un{22}_1$: $\tau_{y, x}=0$. We have $\tau_{y, 1}=1$, $\tau_{1, y}=\sigma_{x, 1}=\sigma_{1, x}:=a$ and $\tau_{x, y}=\sigma_{x, x}=b\not=0$. Therefore,
\begin{eqnarray*}
&&\un{\Delta}(x)=\sigma_{1, 1}1\ot 1 + a1\ot x + ax\ot 1 + bx\ot x,\\
&&\un{\Delta}(y)=a1\ot y + y\ot 1 + bx\ot y.
\end{eqnarray*}
In a manner similar to the one above, we compute that 
\begin{eqnarray*}
\un{\Delta}(x)\un{\Delta}(x)\hspace{-1mm}&=&\hspace{-1mm}
\sigma_{1, 1}1\ot 1+ax_{\un{1}}\ot xx_{\un{2}}+ axx_{\un{1}}\ot x_{\un{2}}+bxx_{\un{1}}\ot xx_{\un{2}}\\
\hspace{-1mm}&=&\hspace{-1mm}
\sigma_{1, 1}^21\ot 1 + 2a\sigma_{1, 1}1\ot x + 2a\sigma_{1, 1}x\ot 1 + 2(a^2+b\sigma_{1, 1})x\ot x + a^21\ot x^2\\
&& 
+ 2abx\ot x^2 + a^2x^2\ot 1 + 2abx^2\ot x + b^2x^2\ot x^2.
\end{eqnarray*}

$\bullet$ If $x^2=0$, $\sigma_{1, 1}=a=0$, so $\un{\Delta}(y)=y\ot 1 + bx\ot y$ and 
$\un{\Delta}(y)\un{\Delta}(y)=y^2\ot 1 + b(\gamma\mfq^{2j+1}+\ov{\gamma})y\ot y$. 
We cannot have $y^2=0$, otherwise $\gamma\mfq^{2j+1}+\ov{\gamma}=0$, which is false for any scalars $\overline\gamma$, $\gamma$. 
Thus $y^2=x$, and since $ \un\Delta(x)=b x\ot x$, we get again a contradiction.

$\bullet$ If $x^2=1$, $a(b+\sigma_{1, 1})=0$, $\sigma_{1, 1}^2 + b^2 + 2a^2=1$ and $a^2+ b\sigma_{1, 1}=0$. Since $y^2=0$, 
$a(1+\mfq^{2j+1}) + b(\mfq^{2j+1}-1)=0$ or, equivalently, $a=-\mfq^{2j+1}b$. Then, as $xy=y$,  
\begin{eqnarray*}
\un{\Delta}(y)&=&\un{\Delta}(x)\un{\Delta}(y)=\sigma_{1, 1}y_{\un{1}}\ot y_{\un{2}} + ay_{\un{1}}+xy_{\un{2}}+axy_{\un{1}}\ot y_{\un{2}}+
bxy_{\un{1}}\ot xy_{\un{2}}\\
&=&(a^2+b^2+ab+b\sigma_{1, 1})1\ot y + (a+\sigma_{1, 1})y\ot 1 + (a^2+2ab+b\sigma_{1, 1})x\ot y + (a+b)y\ot x.
\end{eqnarray*}
Consequently, $b=-a=b\mfq^{2j+1}$, so $b=0$, which is not the case.

$\un{22}_2$: $\tau_{y, x}\not=0$. We have $\tau_{y, x}=\sigma_{x, x}=\tau_{x, y}:=b\not=0$ and 
$\tau_{1, y}=1-\un{\va}(x)\tau_{x, y}=1-\un{\va}(x)\tau_{y, x}=\tau_{y, 1}=\sigma_{x, 1}=\sigma_{1, x}:=a$. Therefore,
\begin{eqnarray*}
&&\un{\Delta}(x)=\sigma_{1, 1}1\ot 1 + a1\ot x + ax\ot 1 + bx\ot x,\\
&&\un{\Delta}(y)=a1\ot y + ay\ot 1 + bx\ot y + by\ot x.
\end{eqnarray*}
The formula for $\un{\Delta}(x)$ is the same as in the subcase $\un{22}_1$. Thus, assuming $x^2=0$, we get $\sigma_{1, 1}=a=0$ and 
$\un{\Delta}(y)\un{\Delta}(y)=b^2\gamma\ov{\gamma}(1+\mfq^{2j+1})y\ot y$. We cannot have $y^2=0$, otherwise, $\gamma=\ov{\gamma}=1$ and $b=0$. 
It follows that $y^2=x$, implying $\gamma=\ov{\gamma}=0$ and $\Delta(x)=\Delta(y)\Delta(y)=0$. So, $b=0$ which is false. 

Hence, necessarily, $x^2=1$, and this fact forces $a(b+\sigma_{1, 1})=0$, $2a^2+ b^2 +\sigma_{1, 1}^2=1$ and $a^2+b\sigma_{1, 1}=0$. 
In particular, we deal with $B_{0o}^{10}$ or $B_{0o}^{10c}$, hence $y^2=0$, $xy=y$ and $yx=\overline\gamma y$ 
with $\ov{\gamma}$ either $1$ or $-1$. The fact that $0=\un{\Delta}(y^2)=\un{\Delta}(y)\un{\Delta}(y)$ 
is equivalent to $a^2=b^2$ if $\overline\gamma=-1$ and to $a=-b$ if $\overline\gamma=1$. Then, by $\un{\Delta}(y)=\un{\Delta}(xy)$ and 
\begin{eqnarray*}
\un{\Delta}(x)\un{\Delta}(y)&=&\sigma_{1, 1}y_{\un{1}}\ot y_{\un{2}}+ ay_{\un{1}}\ot xy_{\un{2}} + axy_{\un{1}}\ot y_{\un{2}}
+ bxy_{\un{1}}\ot xy_{\un{2}}\\
&=&(a\sigma_{1, 1}+ a^2 + ab + b^2)(1\ot y + y\ot 1) + (a^2+2ab+b\sigma_{1, 1})(x\ot y + y\ot x), 
\end{eqnarray*}
we obtain that $a\sigma_{1, 1}+ a^2 + ab + b^2=a$ and $a^2+2ab+b\sigma_{1, 1}=b$. Similarly, $\overline\gamma\un{\Delta}(y)=\un{\Delta}(yx)$ and 
\begin{eqnarray*}
\un{\Delta}(y)\un{\Delta}(x)&=&\sigma_{1, 1}y_{\un{1}}\ot y_{\un{2}} + ay_{\un{1}}\ot y_{\un{2}}x + ay_{\un{1}}((y_{\un{2}})_{[-1]}\cdot x)\ot (y_{\un{2}})_{[0]}
+ by_{\un{1}}((y_{\un{2}})_{[-1]}\cdot x)\ot (y_{\un{2}})_{[0]}x\\
&=&(a\sigma_{1, 1} +\overline\gamma a^2 + ab +\overline\gamma b^2)(1\ot y + y\ot 1) + (b\sigma_{1, 1} +\overline\gamma 2ab + a^2)(x\ot y + y\ot x)
\end{eqnarray*}
entails to $a\sigma_{1, 1} +\overline\gamma a^2 + ab +\overline\gamma b^2=\overline\gamma a$ and $b\sigma_{1, 1} +\overline\gamma 2ab + a^2=\overline\gamma b$. The eight relations that guarantee a braided bialgebra structure on 
$B_{0o}^{10}$ or $B_{0o}^{10c}$ take place simultaneously in the following cases: 

$\bullet$ For $B_{0o}^{10}$, if $a=b\not=0$ then they reduce 
to $\sigma_{1, 1}=-a$, $4a^2=1$ and $2a^2=a$. Thus, $\sigma_{1, 1}=-\frac{1}{2}$ and $a=b=\frac{1}{2}$, leading 
to the braided bialgebra structure on $B_{0o}^{10}$ stated in \equref{1braidbialgB0o10}.  

$\bullet$ For $B_{0o}^{10}$, if $a=-b\not=0$, the above mentioned eight relations reduce to $\sigma_{1, 1}=a$, $4a^2=1$ and $2a^2=a$. Therefore, 
$\sigma_{1, 1}=a=\frac{1}{2}$ and $b=-\frac{1}{2}$, leading to the second braided bialgebra structure on $B_{0o}^{10}$ stated 
in \equref{2braidbialgB0o10}. 

$\bullet$ For $B_{0o}^{10c}$, these relations entail to 
$\sigma_{1, 1}=a=\frac{1}{2}$ and $b=-\frac{1}{2}$. We then get the last braided bialgebra in statement; namely, the one described in \equref{3braidbialgB0o10c}. 

 It remains to show that the three braided bialgebra structures that we just obtained are not isomorphic. Assume the opposite, and take 
$\psi$ a braided bialgebra isomorphism with target bialgebra in (ii) and source bialgebra in (i) or (iii); then $\psi(x)^2=1$ and $\psi(y)^2=0$. 
For $z=a1+bx+cy\in B_{0o}^{10}$ we have $z^2=(a^2+b^2)1+2abx+bc(xy+yx)+2ac=(a^2+b^2)1+2abx+2ac$. Thus, $z^2=1$ if and only if $z\in \{\pm 1, \pm x+cy\}$, and 
$z^2=0$ if and only if $z=c'y$. If the source of $\psi$ is, let's say (i), owing to the definition of the counits we 
get that $\psi(x)=-x+cy$ and $\psi(y)=c'y$, $c'\not=0$. 
Then $\psi(y)=\psi(xy)=\psi(x)\psi(y)$ is equivalent to $c'y=-c'y$, so $c'=0$, and this is false. Likewise, if the source of $\psi$ is 
(iii) then $\psi(x)=x+cy$ and $\psi(y)=c'y$, $c'\not=0$. Thus $c'y=\psi(y)=\psi(yx)=\psi(y)\psi(x)=c'y(x+cy)$ in 
$B_{00}^{10}$, so $c'y=-c'y$. We get $c'=0$, a contradiction.
\end{proof}

If $(B, \un{m}_B, \un{\eta}_B, \un{\Delta}_B, \un{\va}_B)$ is a bialgebra in a braided category $(\Cc, c)$ so is $B^{op+, cop-}$, 
the same object $B$ equipped with the multiplication 
$\un{m}_Bc_{B, B}$, unit $\un{\eta}_B$, comultiplication $c^{-1}_{B, B}\un{\Delta}_B$ and counit $\un{\va}_B$. 

\begin{remarks}
1). For the bialgebra $B_{0o}^{10}$ in (i) above, set $\ov{x}:=\frac{1}{2}(1+x)$. One can see easily that 
\begin{eqnarray}
&&\ov{x}^2=\ov{x},~y^2=0,~\ov{x}y=y,~y\ov{x}=0,\nonumber\\
&&\un{\Delta}(\ov{x})=\ov{x}\ot \ov{x},~\un{\va}(\ov{x})=1,\eqlabel{secbraidedbialg}\\
&&\un{\Delta}(y)=\ov{x}\ot y + y\ot \ov{x},~\un{\va}(y)=0.\nonumber
\end{eqnarray}
We denote by $\un{B}$ the braided bialgebra copy of $B_{0o}^{10}$ in (i), defined by the relations in \equref{secbraidedbialg}.

In the same spirit, for the bialgebra $B_{0o}^{10}$ in (ii) above, set $\ov{x}:=\frac{1}{2}(1-x)$. Then 
\begin{eqnarray}
&&\ov{x}^2=\ov{x},~y^2=0,~\ov{x}y=0,~y\ov{x}=y,\nonumber\\
&&\un{\Delta}(\ov{x})=\ov{x}\ot \ov{x},~\un{\va}(\ov{x})=1,\eqlabel{thirdbraidedbialg}\\
&&\un{\Delta}(y)=\ov{x}\ot y + y\ot \ov{x},~\un{\va}(y)=0.\nonumber
\end{eqnarray}
We denote by $\un{\un{B}}$ the braided bialgebra copy of $B_{0o}^{10}$ in (ii), defined by the relations in \equref{thirdbraidedbialg}. 
Observe that $\un{\un{B}}=\un{B}^{op+, cop-}$ as braided bialgebras. 

Finally, for the bialgebra $B_{0o}^{10c}$ in (iii) above, set again $\ov{x}:=\frac{1}{2}(1-x)$. Then 
\begin{eqnarray}
&&\ov{x}^2=\ov{x},~y^2=0,~\ov{x}y=0,~y\ov{x}=0,\nonumber\\
&&\un{\Delta}(\ov{x})=\ov{x}\ot \ov{x},~\un{\va}(\ov{x})=1,\eqlabel{fourthbraidedbialg}\\
&&\un{\Delta}(y)=\ov{x}\ot y + y\ot \ov{x},~\un{\va}(y)=0.\nonumber
\end{eqnarray}
We denote by $\un{\un{B}}^c$ the braided bialgebra copy of $B_{0o}^{10c}$ in (iii), defined by the relations in \equref{fourthbraidedbialg}. 
Note that $\un{\un{B}}^c=(\un{\un{B}}^c)^{op+, cop-}$ as braided bialgebras. 

2). Neither $\un{B}$, $\un{\un{B}}$ nor $\un{\un{B}}^c$ is a braided Hopf algebra, we leave it to the reader to verify this fact. As a consequence, 
there are no braided Hopf algebras within ${}_\htw^\htw {\cal YD}$, which as objects in ${}_\htw^\htw {\cal YD}$ are isomorphic 
to $M_0\oplus M_{2i}\oplus M_{2j+1}$.
\end{remarks}

\begin{proposition}\prlabel{6 dim quasi-bialgebras}
(i) The biproduct quasi-bialgebra $\un{H}:=\un{B}\times \htw$ can be described as follows. As an algebra, it is unital and generated by 
$G=1\times g$, $X=\ov{x}\times 1$ and $Y=y\times 1$ with relations $G^2=1$, $X^2=X$, $Y^2=0$, $GX=XG$, $GY=-YG$, $XY=Y$ and $YX=0$. The quasi-coalgebra  
structure of $\un{H}$ is determined by the reassociator $\Phi=1-2P_-\ot P_-\ot P_-$, where $P_\pm:=\frac{1}{2}(1\pm G)$, the comultiplication 
\begin{equation}
\Delta(G)=G\ot G,~\Delta(X)=X\ot X,~\Delta(Y)=(P_+ +\mfq^{2j+1}P_-)X\ot Y + Y\ot P_+X - GY\ot P_-X,
\end{equation}
and counit $\va(G)=\va(X)=1$ and $\va(Y)=0$. 

(ii) Similarly, the biproduct quasi-bialgebra $\un{\un{H}}:=\un{\un{B}}\times \htw$ is the unital algebra generated by $G:=1\times g$, $X=1\times x$ 
and $Y=1\times y$ with relations $G^2=1$, $X^2=X$, $Y^2=0$, $GX=XG$, $GY=-YG$, $XY=0$ and $YX=Y$. Its quasi-bialgebra structure is given by the reassociator 
$\Phi=1-2P_-\ot P_-\ot P_-$, comultiplication
\[
\Delta(G)=G\ot G,~\Delta(X)=X\ot X,~\Delta(Y)=(P_+ +\mfq^{2j+1}P_-)X\ot Y + Y\ot P_+X - GY\ot P_-X
\]  
and counit $\va(G)=\va(X)=1$ and $\va(Y)=0$, where $P_\pm:=\frac{1}{2}(1\pm G)$. $\un{\un{H}}$ is the op-version of $\un{H}$.  

(iii) Last but not least, $\un{\un{H}}^c:=\un{\un{B}}^c\times \htw$ is the unital algebra generated by $G:=1\times g$, $X=1\times x$ 
and $Y=1\times y$ with relations $G^2=1$, $X^2=X$, $Y^2=0$, $GX=XG$, $GY=-YG$ and $XY=YX=0$. Its quasi-bialgebra structure is given by the reassociator 
$\Phi=1-2P_-\ot P_-\ot P_-$, comultiplication
\[
\Delta(G)=G\ot G,~\Delta(X)=X\ot X,~\Delta(Y)=(P_+ +\mfq^{2j+1}P_-)X\ot Y + Y\ot P_+X - GY\ot P_-X
\]  
and counit $\va(G)=\va(X)=1$ and $\va(Y)=0$. As before, $P_\pm:=\frac{1}{2}(1\pm G)$. 
\end{proposition}
\begin{proof}
It is left to the reader.
\end{proof}

\begin{proposition}\label{ssqbip rnk3}
    A $3$-dimensional Hopf algebra within the category of  left Yetter-Drinfeld modules over $H(2)$ is  isomorphic, in the notation of 
		Proposition \ref{3dYD}, to either $B_{C_6}$ or $B_*$.
\end{proposition}
\begin{proof}
    It follows from the results proved so far 
		that all $3$-dimensional Hopf algebras in ${}_\htw^\htw {\cal YD}$ are isomorphic, as Yetter-Drinfeld modules, to $M_0\otimes M_{2i}\otimes M_{2j}$.
    In particular, $\htw$ acts trivially on all $3$-dimensional Hopf algebras in ${}_\htw^\htw {\cal YD}$; therefore, the reassociator disappears in all the relations from (\ref{modalg1}) to (\ref{moltcon2}). Consequently, we reduced to the relations that determine the Hopf algebras within ${}_{k[C_2]}^{k[C_2]} {\cal YD}$. 
		This implies that the $3$-dimensional Hopf algebras in ${}_\htw^\htw {\cal YD}$ coincide with those in ${}_{k[C_2]}^{k[C_2]} {\cal YD}$ on which $\htw$ acts trivially.
\end{proof}

\begin{remark}\label{ssbip htw}
     Proposition \ref{ssqbip rnk3} implies that all $6$-dimensional quasi-Hopf algebras with a projection on $\htw$ are semisimple. Indeed, we know from Proposition 
		\ref{3dYD} that the biproduct quasi-Hopf algebras $B_{C_6}\times \htw$ and $B_*\times \htw$ are isomorphic as algebras to $k[C_6]$ and $k^{S_3}$, respectively. 
 \end{remark}

\begin{corollary}\label{bipR3D6}
    All biproduct quasi-Hopf algebras that are defined by $3$-dimensional braided Hopf algebras within ${}_{\htw}^{\htw}{\cal YD}$ are semisimple.
\end{corollary}
\begin{proof}
Let $B\times H$ be a $6$-dimensional biproduct quasi-Hopf algebra that is free of rank $3$ over the $2$-dimensional quasi-Hopf algebra $H$. 
Up to twisting, $H$ is either $k[C_2]$ or $\htw$ (see e.g. \cite[Example 1.16]{bcpvo}). By \cite[Theorem 5.1]{bn}, $B\times H$ is isomorphic to a 
biproduct quasi-Hopf algebra, a free module of rank $3$ over either $k[C_2]$ or $\htw$. 
Then, our assertion follows from Remark \ref{ssbip htw} and the fact that all $6$-dimensional Hopf algebras are semisimple.
\end{proof}

\section{Quasi-Hopf algebras of dimension $6$}\selabel{6dmqHas}
\setcounter{equation}{0}
Let $p<q$ be prime numbers. In \cite[Theorem 6.3]{ego}, Etingof and Gelaki classify the fusion categories of Frobenius-Perron dimension $pq$ with all the 
simple objects having integer dimension. As a byproduct, they obtain also the classification of the $pq$-dimensional semisimple quasi-Hopf algebras in terms of their category of representations. For $p=2$, a special situation occurs in the classification result quoted above; when $pq=6$ this special case is attributed to T. Chmutova.

In this section we prove that any $6$-dimensional quasi-Hopf algebra is semisimple, and this fact complete the classification of quasi-Hopf algebras 
in dimension $6$ (over a field $k$ of characteristic zero and algebraically closed). We get also the classification of the finite tensor categories with Frobenius-Perron dimension $6$ with all the objects having integer Frobenius-Perron dimension. 

To prove the classification results mentioned above we need the lemma below. Recall that a $k$-algebra $A$ is basic if all the irreducible representations of $A$ 
are $1$-dimensional or, equivalently, if $\frac{A}{J_A}\simeq k\times\cdots\times k$ as an algebra, where $J_A$ is the Jacobson radical of $A$.  

\begin{lemma}\lelabel{nonssisbasic}
Any $6$-dimensional non-semisimple quasi-Hopf algebra is basic. 
\end{lemma}
\begin{proof}
By the way of contradiction, assume that there 
is a $6$-dimensional non-semisimple quasi-Hopf algebra that is not basic. Then, if we denote by $J$ the Jacobson radical of $H$, 
$J$ is a non-zero nilpotent ideal of $H$ and $H_0:=\frac{H}{J}$ is a semisimple algebra of dimension $4$ or $5$ (owing to the Wedderburn-Artin theorem). The case 
${\rm dim}_kH_0=4$ cannot occur since $\va$, the counit of $H$, provides a $1$-dimensional (irreducible) representation of $H$, forcing 
$H_0\simeq k^4$; thus $H$ is basic, which is false. 

So, we reduced to the case  ${\rm dim}_kH_0=5$, when $H_0\simeq k\times M_2(k)$ as an algebra; $M_2(k)$ is the  algebra of 
$2$-square matrices. Also,  
${\rm dim}_kJ=1$, and we assume that $J$ is generated by the non-zero element $x$. Last but not least, denote by $p: H\ra H_0$ the canonical projection, 
a $k$-algebra morphism. 

Owing to \cite[Theorem 1.4.9]{abe}, there exists a semisimple subalgebra $\mathbb{S}$ of $H$ such that 
$H=\mathbb{S}\oplus J$ as $k$-vector spaces; obviously, $\mathbb{S}\simeq H_0$ as $k$-vector spaces. Since $J^2=0$, it is easy to check 
(see also \cite{bm}) that $\mathbb{S}\ni s\mapsto p(s)\in H_0$ is a $k$-algebra isomorphism, hence $\mathbb{S}\simeq k\times M_2(k)$ as algebras. 
Clearly, $J$ nilpotent implies $\va(x)=0$. It follows then that $\va$, the counit of $H$, 
restricts to an algebra morphism $\ov{\va}: \mathbb{S}\ra k$: $\va(s\oplus \lambda x)=\ov{\va}(s)$, for all $h=s\oplus \lambda x\in 
H=\mathbb{S}\oplus J$. As $\mathbb{S}\simeq k\times M_2(k)$ has a unique $1$-dimensional representation, it follows that $\ov{\va}$ is the unique 
algebra morphism from $\mathbb{S}$ to $k$, and that ${\rm Ker}(\ov{\va})\simeq M_2(k)$. 
This last fact allows us to see that the multiplication of $H$ can be expressed in terms of the multiplication of 
$\mathbb{S}$ and $x$ as 
\[
(s \oplus \l x)(s'\oplus \l' x)=ss'\oplus (\l'\ov{\va}(s) + \l \ov{\va}(s'))x,
\]
for all $s, s'\in \mathbb{S}$ and $\l x, \l' x\in J$; $\l, \l'$ are arbitrary scalars. We wrote an element in $H=\mathbb{S}\oplus J$ as 
$s\oplus \l x$ in order to distinguish the $\mathbb{S}$ and $J$ components of the elements in $H$. Thus, if we take $\{1, E_{ij}\mid 1\leq i,j\leq 2\}$ 
the canonical basis of $\mathbb{S}$ obtained from the canonical basis of $k\times M_2(k)$ through the algebra isomorphism $\mathbb{S}\simeq k\times M_2(k)$ 
then $\{1, E_{ij}, x\mid 1\leq i, j\leq 2\}$ is a basis of $H$ and the algebra structure of $H$ is determined by the following rules: $1$ is the unit of $H$ 
and, since $\va(x)=0$ and $\va(E_{ij})=0$ for all $1\leq i, j\leq 2$, 
\[
x^2=0,~xE_{ij}=E_{ij}x=0~\mbox{and}~E_{ij}E_{st}=\delta_{j,s}E_{it},~\forall~1\leq i, j, s, t\leq 2.
\]    
 
We look now at $\Delta(x)$, an element in $H\ot H$ which we will write as 
\[
\Delta(x)=1\ot (\l 1 + \sum\limits_{i,j}\l_{ij}E_{ij} + \mu x) + \sum\limits_{i,j}E_{ij}\ot (\l'_{ij}1 + \sum\limits_{s,t}\l_{st}^{ij}E_{st} + \mu_{ij}x) 
+ x\ot (\gamma 1 + \sum\limits_{i,j}\gamma_{ij}E_{ij} + \theta x),
\]
for some scalars $\l, \mu, \l_{ij}, \l'_{ij}, \l_{st}^{ij}, \mu_{ij}, \gamma, \gamma_{ij}$ and $\theta$. 
By using that $\va$ is counit for $\Delta$, we get $\l=0$, $\mu=\gamma=1$ and $\l_{ij}=\l'_{ij}=0$, for all $1\leq i, j\leq 2$. Thus 
\[
\Delta(x)=1\ot x + \sum\limits_{i,j}E_{ij}\ot (\sum\limits_{s,t}\l_{st}^{ij}E_{st} + \mu_{ij}x) 
+ x\ot (1 + \sum\limits_{i,j}\gamma_{ij}E_{ij} + \theta x).
\]  
But $\Delta$ is multiplicative, and therefore $\Delta(x)\Delta(x)=\Delta(x^2)=0$. On the other hand, 
\begin{eqnarray*}
\Delta(x)\Delta(x)&=&2x\ot x + \sum\limits_{i, j}E_{ij}\ot \sum\limits_{s,t,u,v}\l^{iu}_{st}\l_{tv}^{uj}E_{sv},
\end{eqnarray*}   
and thus $\Delta(x^2)$ cannot be zero. We landed in this way to a contradiction, so our proof is finished. 
\end{proof}

Let $H$ be a basic quasi-Hopf algebra which is not semisimple, and take $J$ the Jacobson radical of $H$. It is well known that $J$ is a 
quasi-Hopf ideal of $H$; see for instance \cite[Lemma 1.1]{GrSo}, the proof in the quasi-Hopf case is identical. Thus $H_0=\frac{H}{J}$ 
has a unique quasi-Hopf algebra structure with respect to which the canonical projection $p: H\ra H_0$ becomes a quasi-Hopf algebra morphism.   

Let $A:={\rm gr}(H)=\bigoplus_{m\geq 1}H_m$ be the graded algebra associated to the filtration of $H$ by powers of $J$; here 
$H_m:=\frac{J^m}{J^{m+1}}$, for all $m\geq 0$, where $J^0:=H$. Then $A$ has a canonical quasi-Hopf algebra structure, with reassociator $\Phi$ 
defined by a $3$-cocycle of the group $G$ obtained through algebra isomorphisms $H_0\simeq k\times\cdots\times k\simeq k^G\simeq k[G]$. 
Consequently, when $H$ is not semisimple and basic we have quasi-Hopf algebra morphisms 
$
\xymatrix{
H_0 \ar[r]<2pt>^i &\ar[l]<2pt>^{\pi} {\rm gr}(H)
}
$ 
such that $\pi i=\Id_{H_0}$, hence ${\rm gr}(H)$ is a biproduct quasi-Hopf algebra between a braided Hopf algebra  
$B$ in ${}_{k_\Phi[G]}^{k_\Phi[G]}{\cal YD}$ and $k_\Phi[G]$. Furthermore, each homogenous component $H_m$ of $H_0$ can be seen as an object 
in the category of two-sided two-cosided Hopf module over $H_0$. Since any object of this category is free as a left and right $H_0$-module it follows that so 
is every $H_m$; see also \cite{sch}.    

\begin{theorem}\thlabel{6ss}
    Any $6$-dimensional quasi-Hopf algebra is semisimple.
\end{theorem}
\begin{proof}
Assume that there exists a non-semisimple quasi-Hopf algebra of dimension $6$. \leref{nonssisbasic} guarantees us that $H$ is basic and then we can apply 
all the results mentioned above. We distinguish the following cases.  
    
$\un{Case~ 1}$: ${\rm dim}_kJ=1$. We have ${\rm dim}_kH_0=5$ and $J^2=0$, thus $H_1$ cannot be free over $H_0$.   
    
$\un{\rm Case~ 2}$: ${\rm dim}_kJ=2$. Similarly, ${\rm dim}_kH_0=4$ and since ${\rm dim}_kH_1$ is at most $2$ and so it cannot be a non-zero power of $4$, 
$H_1$ is not free over $H_0$.

$\un{\rm Case~ 3}$: ${\rm dim}_kJ=3$. As ${\rm dim}_kH_0=3$, we must have $J^2=0$. Therefore, ${\rm gr}(H)=H_0\oplus J$ has dimension $6$ and is a biproduct 
quasi-Hopf algebra between a braided Hopf algebra $B$ in ${}_{k_\Phi[C_3]}^{k_\Phi[C_3]}{\cal YD}$ and $k_\Phi[C_3]$, for a certain $3$-cocycle $\Phi$ 
of $k[C_3]$. Then $B$ has dimension $2$ and by the results obtained in \cite{bm} it follows that ${\rm gr}(H)$ is semisimple. This is false because 
in this particular case the Jacobson radical of ${\rm gr}(H)$ is $J$ and $J$ is not the null ideal. 

$\un{\rm Case~ 4}$: ${\rm dim}_kJ=4$. We have ${\rm dim}_kH_0=2$, which imposes ${\rm dim}_kJ^2\in \{0, 2\}$. If ${\rm dim}_kJ^2=0$, 
${\rm gr}(H)=H_0\oplus J$ has again dimension $6$ and it is a biproduct between a braided Hopf algebra 
$B$ in ${}_{k_\Phi[C_2]}^{k_\Phi[C_2]}{\cal YD}$ and $k_\Phi[C_2]$, for a certain $3$-cocycle $\Phi$ 
of $k[C_2]$. Then $B$ has dimension $3$ and by the results obtained in the previous section it follows that ${\rm gr}(H)$ is semisimple. But 
the Jacobson radical of ${\rm gr}(H)$ is $J\not=0$, a contradiction. We also reach the same contradiction when ${\rm dim}_kJ^2=2$, because  
$J^3=0$ and so ${\rm gr}(H)=H_0\oplus H_1\oplus H_2$ is of dimension $6$, too. 
 
$\un{\rm Case~ 5}$: ${\rm dim}_kJ=5$. Since $H_0\simeq k$ is the trivial Hopf algebra, ${\rm gr}(H)$ is an ordinary Hopf algebra. By dimension counting,
 it is straightforward, although it is a bit lengthy, to check that ${\rm gr}(H)$ has dimension $6$ in any case. 
It follows that ${\rm gr}(H)$ is a semisimple Hopf algebra. On the other hand, the Jacobson radical of ${\rm gr}(H)$ has dimension $5$, 
so it cannot be the null ideal.  

So our assumption is false, therefore any quasi-Hopf algebra of dimension $6$ is semisimple.     
    \end{proof}

\begin{remark}
    Since the quasi-bialgebras in \prref{6 dim quasi-bialgebras} are all non-semisimple algebras, it follows that they do not admit a quasi-Hopf algebra structure.
\end{remark}

    We translate our result in the language of tensor categories. For the definition of the Frobenius-Perron dimension and of a tensor or 
		fusion category we refer to \cite{ce}.
		
    \begin{corollary}
        Let ${\cal C}$ be a finite tensor category of Frobenius-Perron dimension $6$ such that every object of  ${\cal C}$ has integer Frobenius-Perron dimension. Then ${\cal C}$ is a fusion category.
    \end{corollary}
    \begin{proof}
        It is a direct consequence of \thref{6ss} and \cite[Proposition 2.6]{eo}.
    \end{proof}

We end the paper by giving an explicit presentation for all $6$-dimensional quasi-Hopf algebras. Our presentation is 
in concordance with the classification of fusion categories produced in \cite{ego}.

Let $\xi$ be a $n$th root of unity and $\sigma$ a generator for the cyclic group of order $n$, $C_n$. For $a, n$ integers with $n$ non-zero, 
denote by $\lfloor\frac{a}{n}\rfloor$ the integer part of the rational number $\frac{a}{n}$. Then, 
it is well known that $3$-cocycle representatives for the group $C_n$ are given by $\{\phi_{\zeta^a}\mid 0\leq a\leq n-1\}$, where 
    \begin{eqnarray}\eqlabel{cn3cc}
        && \phi_{\zeta^a}(\sigma^i,\sigma^j,\sigma^l)= \zeta^{ai\lfloor\frac{j+l}{n} \rfloor },~\forall~0\leq i, j, l\leq n-1. 
    \end{eqnarray}

As far as we are concerned, $3$-cocycle representatives for the group $S_3$ can be deduced from the formula (3.2.8) in \cite{wp}. Namely, they are given by 
\begin{equation}\eqlabel{cn3S3}
    \omega_p = \sum_{I,J,L=0}^1\sum_{i,j,l=0}^2 \mfq^{p(-1)^{J+L}i\lfloor \frac{(-1)^Lj + l}{3}\rfloor} (-1)^{pIJL},
\end{equation}   
where $\mfq$ is a primitive $3$th root of unity and $p\in\{0,\cdots,5\}$. This fact completes \cite[Corollary 4.2]{ego}, 
where it is shown that $H^3(S_3, \mathbb{C}^*)\simeq C_6$. 

Consider the group $S_3$ generated by $\tau$ and $\sigma$; $\tau$ is a transposition, while $\sigma$ is a cycle of length $3$. 

    \begin{theorem}\thlabel{fulldes6dimqHas}
        Let $\mfq$ be a primitive $3$th root of unity and $\xi$ a primitive $6$th root of unity. Then, any $6$ dimensional quasi-Hopf algebra 
				is twist equivalent to one of the following:
        
        $\bullet$ $(k[C_6],\Phi_a)$, the group Hopf algebra $k[C_6]$, seen as a quasi-Hopf algebra with reassociator
        \begin{equation}\eqlabel{reassC6}
            \Phi_a=\sum_{i,j,l=0}^5 \xi^{ai\lfloor \frac{j+l}{3}\rfloor} 1_i\ot 1_j\ot 1_l,
        \end{equation} 
        where, if $g$ generates $C_6$, $\{1_i:=\frac{1}{6}\sum_{t=0}^5\xi^{-ti}g ^t\mid 0\leq i\leq 5\}$ is the set of primitive idempotents of 
				$C_6$. There are $6$ quasi-Hopf algebras of this type which are not twist equivalent, given by $0\leq a \leq 5$.

        $\bullet$ $(k^{S_3}, \Phi_p)$, the dual of the group Hopf algebra $k[S_3]$ endowed with the reassociator 
        \begin{equation}\eqlabel{reassfunS3}
            \Phi_p = \sum_{I,J,L=0}^1\sum_{i,j,l=0}^2 \mfq^{p(-1)^{J+L}i\lfloor \frac{(-1)^Lj + l}{3}\rfloor} (-1)^{pIJL}\ 
						P_{\tau^I\sigma^i}\ot P_{\tau^J\sigma^j}\ot P_{\tau^L\sigma^l}.
        \end{equation}
  Here $\{P_{\tau^I\sigma^i}\mid 0\leq I\leq 1, 0\leq i\leq 2\}$ is the basis of $k^{S_3}$ dual to the basis 
$\{\tau^I\sigma^i\mid 0\leq I\leq 1, 0\leq i\leq 2\}$ of $k[S_3]$. There are $6$ quasi-Hopf algebras of this type which are not twist equivalent, 
given by $0\leq p\leq 5$.

        $\bullet$ The group Hopf algebra $k[S_3]$.

        $\bullet$ $(k[S_3],\Psi_a)$, the group Hopf algebra $k[S_3]$ with reassociator $\Psi_a$ obtained from a non-trivial $3$-cocycle of 
				$C_3=\left<\sigma\right>$ which commutes with $\tau\ot \tau \ot \tau$; namely,  
        \begin{equation}\eqlabel{specreass}
            \Psi_a = \sum_{i,j,l=0}^2 \mfq^{ai\lfloor \frac{j+l}{3}\rfloor+ ai(\lfloor \frac{(j+l)'}{2}\rfloor - \lfloor \frac{j}{2}\rfloor - 
						\lfloor \frac{l}{2}\rfloor)} \ {\bf 1}_i\ot {\bf 1}_j\ot {\bf 1}_l,
        \end{equation}
        where $\{{\bf 1}_i:=\frac{1}{3}\sum_{t=0}^2\mfq^{-ti}\sigma^t\mid 0\leq i\leq 2\}$ is the set of primitive idempotents of 
				$C_3=\le\sigma\ri$ and $(m)'$ denotes the remainder of the division of the integer $m$ by $3$. There are $2$ quasi-Hopf algebras of this type 
				which are not twist equivalent, defined by $a\in \{1, 2\}$,
    \end{theorem}
    \begin{proof}
        Follows from \thref{6ss} and  \cite[Theorem 6.3]{ego}. 
				To get \equref{reassC6} we have identified the Hopf algebra of functions on $C_6$, $k^{C_6}$, with the group Hopf algebra $k[C_6]$ by using the 
				isomorphism provided by the correspondence $P_{g^i}\mapsto 1_i$, $0\leq i\leq 5$; see \cite[Proposition 3.57]{bcpvo}. Then we made use of 
				\equref{cn3cc} to get for free all the non-twist equivalent structures on $k[C_6]$. 
				
				Likewise, to get \equref{reassfunS3} we used the fact that $k[S_3]$ is a dual quasi-Hopf algebra 
				(for the definition see for instance \cite[Section 3.7]{bcpvo}) with reassociator defined by \equref{cn3S3}. Then passing to the dual of it 
				we obtain $k^{S_3}$ with the reassociator mentioned in \equref{reassfunS3}.
				
				The only thing left to prove is the fact that $\psi_a$, $1\leq a\leq 2$, defined by 
        \[
            \psi_a(\sigma^i,\sigma^j,\sigma^l)=\mfq^{ai\lfloor \frac{j+l}{3}\rfloor+ a(\lfloor \frac{(j+l)'}{2}\rfloor - \lfloor \frac{j}{2}\rfloor - \lfloor \frac{l}{2}\rfloor)},
        \]
         is a $3$-cocycle for the group $C_3$ and, moreover, it is invariant under the action of $C_2$; the latter means that 
				$\psi_a(\sigma^i,\sigma^j,\sigma^l)=\psi_a(\sigma^{-i},\sigma^{-j},\sigma^{-l})$, for all $0\leq i, j, l\leq 2$. 
				Towards this end, we show that $\psi_a$ is cohomologus to $\phi_{\mfq^a}$. Indeed, if we take $g_a: C_3\times C_3\ra k$ defined by 
				$g_a(\sigma^i,\sigma^j):=\mfq^{ai\lfloor \frac{j}{2}\rfloor}$, for all $0\leq i, j\leq 2$, then 
         \begin{eqnarray*}
             &&  \phi_{\mfq^a}(\sigma^i,\sigma^j,\sigma^l) {g_a(\sigma^{j},\sigma^{l} )}{g_a(\sigma^{(i+j)'},\sigma^{l} )}^{-1}{g_a(\sigma^{i},\sigma^{(j+l)'} )}{g_a(\sigma^{i},\sigma^{j} )}^{-1}\\
             &&\hspace{1cm}= \mfq^{ai\lfloor\frac{j+l}{3}\rfloor}\mfq^{ai\lfloor \frac{j}{2}\rfloor} \ \mfq^{aj\lfloor \frac{l}{2}\rfloor}\ 
             \mfq^{-a(i+j)\lfloor \frac{l}{2}\rfloor}\  
             \mfq^{ai\lfloor \frac{(j+l)'}{2}\rfloor}\  
             \mfq^{-ai\lfloor \frac{j}{2}\rfloor} = \psi_a(\sigma^i,\sigma^j,\sigma^l).
         \end{eqnarray*}
         Moreover, $\psi_a(\sigma,\sigma,\sigma)=\mfq^a=\psi_a(\sigma^2,\sigma^2,\sigma^2)$ , $\psi_a(\sigma,\sigma,\sigma^2)=1=\psi_a(\sigma^2,\sigma^2,\sigma)$, 
				$\psi_a(\sigma,\sigma^2,\sigma)=1=\psi_a(\sigma^2,\sigma,\sigma^2)$, $\psi_a(\sigma^2,\sigma,\sigma)=\mfq^{2a}=\psi_a(\sigma,\sigma^2,\sigma^2)$. 
				The rest of the invariance conditions for $\psi_a$ follow from $\psi_a$ being normalized, as $g_a(\sigma^i,1)=1=g_a(1,\sigma^j)$ for any $i$ and $j$.
    \end{proof}
    
    \begin{remark}
        $\Psi_a$ produces a quasi-Hopf algebra structure on $k[\left<\sigma\right>]\simeq k[C_3]$, twist equivalent to the one produced by 
				$\Phi_{\mfq^a}$ on $k[C_3]$. In addition, the natural embedding of $k[C_3]$ in $k[S_3]$ can be seen as a quasi-Hopf algebra embedding of  
				$(k[C_3], \Psi_a)$ into $(k[S_3], \Psi_a)$. But there is no quasi-Hopf projection covering this embedding, because $(k[S_3], \Psi_a)$ cannot be described 
				as a biproduct quasi-Hopf algebra.  
    \end{remark}


\end{document}